\documentclass[11pt]{amsart}
\usepackage{amssymb,amsthm,amsmath,amstext}
\usepackage{mathrsfs}  % allows nice script font for sheaves
\usepackage{bm}        % allows bold italic letters in math mode
\usepackage{mathtools} % allows more extendible arrows
\usepackage{color}
\usepackage[bbgreekl]{mathbbol}
\usepackage{multirow}
\usepackage{enumitem}
\usepackage[noadjust]{cite}
\usepackage{tabularx}
\usepackage{blkarray}
\usepackage{placeins}

\usepackage[left=3cm,right=3cm]{geometry}

\usepackage{graphicx}
\usepackage{array}
\usepackage{booktabs}
\usepackage[all]{xy}
\usepackage{tikz}
\usepackage{tikz-cd}

\theoremstyle{plain}
\newtheorem{theorem}{Theorem}[section]

\newtheorem{proposition}[theorem]{Proposition}
\newtheorem{lemma}[theorem]{Lemma}

\numberwithin{equation}{section}

\theoremstyle{definition}
\newtheorem{definition}[theorem]{Definition}

\newtheorem{example}[theorem]{Example}

\theoremstyle{remark}
\newtheorem{remark}[theorem]{Remark}

\newcommand{\sheaf}[1]{\mathscr{#1}}

\newcommand{\OO}{\sheaf{O}}

\newcommand{\Z}{\mathbb Z}

\newcommand{\C}{\mathbb C}
\renewcommand{\P}{\mathbb P}
\newcommand{\Q}{\mathbb Q}

\usepackage[backref=page]{hyperref}
\newcommand{\bP}{\mathbb{P}}

\newcommand{\git}{/\kern-0.2em/}

\newcommand{\Db}{\mathrm{D}^{\mathrm{b}}}
\newcommand{\Ku}{\mathrm{Ku}}

\begin{document}

\title[Naive atoms of blowups: examples]{Naive atoms of blowups: examples}

\author[B\"ohning]{Christian B\"ohning}\thanks{For the purpose of open access, the authors have applied a Creative Commons Attribution (CC BY) licence to any Author Accepted Manuscript version arising from this submission.}
\address{Christian B\"ohning, Mathematics Institute, University of Warwick\\
Coventry CV4 7AL, England}
\email{C.Boehning@warwick.ac.uk}

\author[von Bothmer]{Hans-Christian Graf von Bothmer}
\address{Hans-Christian Graf von Bothmer, Fachbereich Mathematik der Universit\"at Hamburg\\
Bundesstra\ss e 55\\
20146 Hamburg, Germany}
\email{hans.christian.v.bothmer@uni-hamburg.de}

\author[Su'a]{Zac Su'a}
\address{Zac Su'a, Mathematics Institute, University of Warwick\\
Coventry CV4 7AL, England}
\email{Zac.Su-A@warwick.ac.uk}

%\author[Marquand]{Lisa Marquand}
%\address{Lisa Marquand, Courant Institute of Mathematical Sciences, New York University\\
%251 Mercer Street\\
%NY 10012, USA}
%\email{lisa.marquand@nyu.edu}

\date{\today}

% \subjclass[2010]{11E08, 11E20, 11E88, 14D06, 14F22, 14L35, 15A66, 16H05}

\begin{abstract}
We define naive atomic decompositions of smooth projective varieties. We show that they satisfy a naive version of Iritani's blowup formula in several examples that are complicated enough to show most interesting features of the general theory while being simple enough to be computable by elementary methods. 
\end{abstract}

\maketitle 

\section{Introduction}\label{sIntroduction}

The theory of atoms and equivariant atoms \cite{KKPY25, CKK25} is a seminal development in birational geometry and has already had profound influence on the field \cite{ESS25, BGMP26, Fay26}. In these foundational articles, the atoms are obtained from the spectrum of the quantum multiplication by the Euler vector field in the big (algebraic) quantum cohomology. Unfortunately, computations in the big quantum cohomology are notoriously difficult in general. Therefore, one would like to restrict to a computation with the small quantum cohomology. Quantum multiplication by the Euler vector field then restricts to small quantum multiplication by the anticanonical bundle. Thus we make %We work over the complex numbers throughout. 

\begin{definition}\label{dNaiveAtoms}
For $X$ a smooth projective Fano variety,  
let $(-K_X)\star$ be the operator, on the algebraic cohomology of $X$, of quantum multiplication by the anticanonical bundle in the small sense. 
We define \emph{the naive atomic decomposition of $X$} as the ordered sequence of multiplicities of eigenvalues of $(-K_X)\star$, for generic values of the Novikov parameters. 

 Moreover, if $QH^{\bullet}_{\mathrm{alg}} (X)$ denotes the small algebraic quantum cohomology algebra, we define the \emph{small atomic decomposition} of $X$ as the ordered sequence of lengths of connected components of the Artinian scheme  $\mathrm{Spec}\, QH^{\bullet}_{\mathrm{alg}} (X)$, for generic values of the Novikov parameters. 
\end{definition}

The small atomic decomposition is a refinement of the naive atomic decomposition. 
We will mostly be concerned with computing naive atomic decompositions, which in most cases agree with the small atomic decompositions. But in some cases there are subtle differences, cf. Example \ref{eQuadric4}. 

In particular, in the notation of \cite{KKPY25}, the sequence $(\rho_{\alpha})_{\alpha}$ for $\alpha$ running over the Hodge atoms of $X$ and $\rho_{\alpha}$ the associated dimension of the space of Hodge classes, is in general a refinement of the naive atomic decomposition of $X$ (but as we will see also often equal to it). Here we do not consider additional invariants such as the Hodge polynomial $P_{\alpha}$ of an atom $\alpha$ although they are essential in irrationality, and in the equivariant birational setting, non-linearisability proofs. Indeed, one motivation for us to embark on this project was to get a better understanding, through computational experiments and concrete examples, of the mechanisms behind Iritani's blowup formula \cite{Iritani23}, which we consider to be the basic miracle in the theory of atoms (previously verified in the surface case in \cite{GS25}). Accordingly, we restrict our attention to the most immediately available numerical data that can be extracted from the atoms of a variety, and track its behaviour under blowups.

\medskip

Our experiments highlight several surprising features (and we believe provide some experimental evidence for) \cite{Iritani23} and \cite{KKPY25}. In the most optimistic scenario, they might also provide some data points for subsequent simplifications of proofs of foundational results of the theory. Specifically, we compute the naive atomic decomposition for the following examples:

\begin{enumerate}
\item 
Smooth cubic fourfolds (with arbitrary rank $r$ of the space of codimension $2$ cycles) in Section \ref{sCubicFourfolds}.
Here the naive atomic decomposition is $(1+r, 1, 1, 1)$.
\item 
A complete intersection of a quadric and cubic in $\mathbb{P}^5$ in Section \ref{sCI23}.  This is particularly interesting from the point of view of derived categories and Kontsevich's conjectural canonical semiorthogonal decompositions: in this case there are two natural semiorthogonal decompositions that are conjecturally not in the same mutation orbit, and we explain how this ambiguity gets resolved from the (naive) atomic point of view.
\item 
Smooth Picard rank $1$ Fano threefolds in 
Section \ref{sRankOneFanoThreefolds}. In all cases the naive atomic decompositions agree  with what is expected from  known semiorthogonal decompositions of the derived categories. Quantum cohomology of these Fanos is a very classical subject studied by lots of people \cite{Gol07}, \cite{Pr06},  \cite{Pr07a, Pr07b}
\item
Verra fourfolds in Section \ref{sVerraFourfolds}. This is interesting because it is an example with Picard rank $2$ which is not a blowup. 
\end{enumerate}

The following variants of Iritani's blowup formula \cite{Iritani23} in this setting are useful. %We write $\mathrm{N}_1 (Y, \mathbb{Z}) = \mathrm{im}(\mathrm{CH}_1(Y) \to H_2(Y, \mathbb{Q}))$ for a smooth projective variety $Y$. 

\begin{definition}\label{dIritani}
Let $\widetilde{X}$ be the blowup of $X$ in a smooth codimension $c$ subvariety $Z$. 

\begin{enumerate}
\item
We say that \emph{the naive variant of Iritani's theorem holds for $(X, Z)$}  if the naive atomic decomposition of $\widetilde{X}$ is the sequence obtained by concatenating and sorting the naive atomic decompositions of $X$ and $(c-1)$ times the naive atomic decomposition of $Z$. 
\item 
We say that \emph{the small variant of Iritani's theorem holds for $(X, Z)$}  if the small atomic decomposition of $\widetilde{X}$ is the sequence obtained by concatenating and sorting the small atomic decompositions of $X$ and $(c-1)$ times the small  atomic decomposition of $Z$. 
\end{enumerate}

%Suppose that the inclusion $Z \hookrightarrow X$ induces an embedding $\mathrm{N}_1 (Z, \mathbb{Z}) \hookrightarrow \mathrm{N}_1 (X, \mathbb{Z})$. Then the {\color{red} small} atomic decomposition of $\widetilde{X}$ is the sequence obtained by concatenating and sorting the {\color{red} small} atomic decompositions of $X$ and $(c-1)$ times the {\color{red} small} atomic decomposition of $Z$. 
\end{definition}

We check that the naive version of Iritani's theorem holds in the following cases: 

\begin{enumerate}
\item 
The blowup of a cubic fourfold in a plane in Section \ref{sPlane}. We show that the matrices representing $(-K_X)\star$ and $(-K_Z)\star$ can be obtained from $(-K_{\widetilde{X}})\star$ as suitable limits as expected. 
This is an example where $X$ has a high multiplicity entry in its naive atomic decomposition, and this shows up again on $\widetilde{X}$. This is remarkable because  a multiple root of a polynomial is generically not preserved by a deformation of the polynomial. 
\item
The blowup of $X=\mathbb{P}^4$ in a complete intersection curve $Z$ of multidegree $(2,2,2)$
in Section \ref{sP4BlownUp}. We observe that the high multiplicity entry in the naive atomic decomposition of the blowup centre is preserved and appears again twice in the naive  atomic decomposition of the blowup. Furthermore, we write the characteristic polynomials of $(-K_X)\star$ and $(-K_Z)\star$ as limits of $(-K_{\widetilde{X}})\star$.
\item
The blowup of a cubic fourfold in a line in 
Section \ref{sLine}. The limiting process to recover the characteristic polynomial of the blowup centre is more intricate here than in the previous cases. 
\item 
The blowup of a cubic threefold in a plane elliptic curve in Section \ref{sElliptic},   illustrating similar features to the ones already described. 
\end{enumerate}

In Section \ref{sCIQuadrics} we show that neither the naive nor the small variant of Iritani's theorem hold in general: indeed, if one blows up $\P^4$ in a Castelnuovo surface, which is isomorphic to $\P^2$ blown up in eight points, the naive and small atomic decomposition of $\P^4$ is $(1,1,1,1,1)$, that of the Castelnuovo surface consists of eleven ones in both cases, but the blowup $Y$ has naive decomposition $9$ and seven ones. Notice that $Y$ can also be obtained by blowing up the complete intersection $X$ of two quadrics in $\P^6$ in a line. $X$ has naive and small atomic decomposition $(9,1,1,1)$, and the line has $(1,1)$. For this blowup the naive variant of Iritani's theorem holds. The small variant of Iritani's theorem cannot hold for both $Y\to X$ and $Y \to \P^4$.

We write $\mathrm{N}_1 (Y, \mathbb{Z}) = \mathrm{im}(\mathrm{CH}_1(Y) \to H_2(Y, \mathbb{Q}))$ for a smooth projective variety $Y$. The failure of the naive and small variants of Iritani's theorem in this example seems to be due to the fact that the inclusion $Z \hookrightarrow X$ does not induce an embedding $\mathrm{N}_1 (Z, \mathbb{Z}) \hookrightarrow \mathrm{N}_1 (X, \mathbb{Z})$. If on the contrary one assumes that it does, then we know of no counterexample to the small variant of Iritani's theorem. However, the naive version can fail for other reasons, namely that the naive atomic decompositions do not coincide with the small atomic decompositions in general, and the small ones appear to be much better behaved, see Example \ref{eQuadric4}.

The Sections in the paper are ordered according to increasing complexity of the methods necessary to treat the examples discussed in them. 

\medskip

We would like to thank Yuri Tschinkel and Zhijia Zhang for many helpful discussions on the subject. Christian B\"ohning would like to thank the Simons Foundation for its hospitality and for providing excellent working conditions during several stays in 2025 and 2026 when this work took shape. Zac Su'a is supported by the Warwick Mathematics Institute Centre for Doctoral Training, and gratefully acknowledges funding from the University of Warwick.

The authors would also like to acknowledge the help of ChatGPT and Gemini in computations and literature research. All suggestions made by ChatGPT and Gemini were checked carefully by the
authors, and all computational results were independently verified in
Macaulay2. Apart from one consistent
sign mistake in the use of the Chern classes of the normal bundle, all of
these computations were correct.

We would like to thank Victor Przyjalkowski for lots of very thoughtful comments and remarks on the first version of this article, and \'{A}d\'{a}m Gyenge for pointing us to a useful reference. 

%The purpose of the present work is to compute atoms for various classes of Fano manifolds of Picard rank $1$ and $2$ and illustrate remarkable features of the general theory, in particular, Iritani's theorem, in these examples. A recurring motive that perhaps has not been emphasised enough so far is that the commutativity relations for the quantum product by divisor classes already gives a lot of information in Picard rank $2$ when combined with information given by the quantum period and the recurrence relations obtained from it. This may also eventually help to understand special cases of Iritani's theorem \cite{Iritani23} more simply. In addition we believe that the computations in this paper can be helpful to obtain non-rationality or non-linearisability results for broader classes of examples in subsequent work. 

%{\color{red} The authors would like to acknowledge the help of ChatGPT and Gemini. Since we were novices in the theory of atoms, we had the basics explained to us by these programs. 
%The first versions of all intersection rings were suggested by Gemini, which
%also generated the corresponding Macaulay2 code. Apart from one consistent
%sign mistake in the use of the Chern classes of the normal bundle, all of
%these computations were correct.
%ChatGPT was also used to improve the English in some parts of the manuscript.
%All suggestions made by ChatGPT and Gemini were checked carefully by the
%authors, and all computational results were independently verified in
%Macaulay2.}

\section{Basic set-up and computational rules}\label{sRules}

We work over the complex numbers $\mathbb{C}$. 

Let $X$ be a smooth projective variety that is Fano (the Fano assumption is not essential, but will simplify the discussion as we will point out below). 

We denote by $\mathrm{NE}(X, \mathbb{Z})\subset \mathrm{N}_1 (X, \mathbb{Z}) = \mathrm{im}(\mathrm{CH}_1(X) \to H_2(X, \mathbb{Q}))$ the monoid of numerically effective algebraic curve classes $\beta$ on $X$. For (homogeneous) cycle classes $A, B, C$ in the space $A^{\bullet}(X)_{\mathbb{Q}}$ of algebraic cycles modulo homological equivalence with rational coefficients we write 
\[
\langle A, B, C\rangle_{\beta}
\]
for the three point genus zero Gromov-Witten invariant, as well as 
\[
\langle A, B \rangle_{\beta}, \quad \langle A \rangle_{\beta}
\]
for its two point and one point variants, and refer to \cite{KM94, FP95, Manin99} for a rigorous definition via integrals over suitably compactified moduli spaces (stacks) of stable maps. Recall that, naively, $\langle A, B, C\rangle_{\beta}$ counts the number of regular maps
\[
f \colon (\mathbb{P}^1 ; 0, 1, \infty) \to X 
\]
from a $\mathbb{P}^1$ with three marked points $0, 1, \infty$ into $X$ such that $f_* ([\mathbb{P}^1]) = [\beta]$ and $f(0) \in \widetilde{A}$, $f(1) \in \widetilde{B}$, $f(\infty) \in \widetilde{C}$ where $\widetilde{A}, \widetilde{B}, \widetilde{C}$ are irreducible subvarieties representing $A, B, C$ and chosen suitably general so that in particular the count comes out finite (of course it is not always possible to make such choices for any classes $A, B, C$; also we ignore subtleties arising from multiplicities in this naive picture). Similarly, $\langle A, B\rangle_{\beta}$ and $\langle A\rangle_{\beta}$ naively count maps 
\[
f \colon (\mathbb{P}^1 ; 0, 1 ) \to X \; \mathrm{and} \; f \colon (\mathbb{P}^1 ; 0 ) \to X
\]
of homology class $\beta$ with $f(0) \in \widetilde{A}$, $f(1) \in \widetilde{B}$, modulo automorphisms of $(\mathbb{P}^1 ; 0, 1 )$ and $(\mathbb{P}^1 ; 0 )$. 

On the few occasions when we have to compute Gromov-Witten invariants from scratch below, these naive counts turn out to be justified, but usually we make heavy use of the following computational rules to calculate more complicated invariants.

We fix a homogeneous basis $T_1, \dots , T_N$ of $A^{\bullet}(X)_{\mathbb{Q}}$ and denote by $T_1^* , \dots , T_N^* \in A^{\bullet}(X)_{\mathbb{Q}}$ the degree-wise dual basis with respect to the intersection pairing between cycles of complementary dimensions. 

\begin{enumerate}
\item
\textbf{Symmetry and $\mathbb{Z}$-multilinearity.}
The expression $\langle A, B, C\rangle_{\beta}$ is symmetric and $\mathbb{Z}$-multilinear in the entries $A, B, C$.
\item 
\textbf{Selection rule.} We have that $\langle A, B, C\rangle_{\beta}\neq 0$ only if 
\[
\mathrm{codim}\, A + \mathrm{codim}\, B + \mathrm{codim}\, C = \dim X + \deg_{(-K_X)} (\beta ) = \dim X + (-K_X).\beta .
\]
\item 
\textbf{Divisor axiom.} If $D$ is the class of a divisor
\[
\langle D , B, C\rangle_{\beta} = (D.\beta) \langle B, C\rangle_{\beta}. 
\]
\item \textbf{The WDVV (Witten-Dijkgraaf-Verlinde-Verlinde) equations.} These are great for computing Gromov-Witten invariants recursively, though their structure does often not give a lot of conceptual insight into how Gromov-Witten invariants are formed. We will mostly use easier methods (namely, the commutativity of the quantum multiplication by two divisor classes), but list the WDVV equations here for completeness: Fix $T_i, T_j, T_k, T_l$ and $\beta$, then 
\[
\sum_{\beta = \beta' + \beta''}  \sum_{a=1}^N \langle T_i, T_j, T_a \rangle_{\beta'} \langle T_a^*, T_k, T_l \rangle_{\beta''} = \sum_{\beta = \beta' + \beta''}  \sum_{a=1}^N \langle T_i, T_k, T_a \rangle_{\beta'} \langle T_a^*, T_j, T_l \rangle_{\beta''} .
\]
\end{enumerate}

Let $\Lambda = \mathbb{C} [ \mathrm{NE}(X, \Z)]$ be the monoid ring generated by $q^{\beta}$, $\beta \in \mathrm{NE}(X, \Z)$ (other coefficients than $\mathbb{C}$ are possible and sometimes useful). 

\begin{definition}\label{dSmallQuantumCohomology}
The small quantum cohomology deforms the product given by the intersection pairing on $A^{\bullet}(X)_{\mathbb{Q}}$ to a product on $A^{\bullet}(X)_{\mathbb{Q}}\otimes \Lambda$. We put
\[
\langle\langle A,B,C \rangle\rangle := \sum_{\beta} \langle A, B, C \rangle_{\beta} q^{\beta}. 
\]
This is known to be finite for $X$ Fano. We then define a new ring structure via
\[
T_i \star T_j = \sum_{k} \langle\langle T_i, T_j, T_k^*  \rangle\rangle T_k .
\]
This turns $A^{\bullet}(X)_{\mathbb{Q}}\otimes \Lambda$ into a commutative and associative ring. We call this $QH^{\bullet}_{\mathrm{alg}} (X)$, the small algebraic quantum cohomology of $X$. 
\end{definition}

\begin{definition}\label{dMatrixQuantumMult}
For any homogeneous cohomology class $D$ we can define the matrix 
\[
	M_{(X, D)}  = (\langle\langle D , T_i, T_j^* \rangle\rangle )_{i,j}  .
\]
This describes the quantum multiplication by $D$ on the space $A^{\bullet}(X)_{\mathbb{Q}}$ with respect to the chosen basis $T_1, \dots , T_N$. 

We will most often use this if $D$ is a divisor. 

In the sequel we will often write $M(X)$ for   $M_{(X, -K_X)}$ to simplify notation. 
\end{definition}

\noindent
By construction, for divisors $D_1, D_2$, 
\[
M_{(X, D_1)} M_{(X, D_2)}  = M_{(X, D_1\star D_2)}.
\]
In particular, since quantum multiplication $\star$ is commutative, we have
\[
	M_{(X, D_1)} M_{(X, D_2)} = M_{(X, D_2)} M_{(X, D_1)}.
\]
If $D$ is a divisor, we have, using the divisor axiom:
\begin{align*}
	M_{(X, D)} 
	&= (\langle\langle D , T_i, T_j^* \rangle\rangle )_{i,j}  \\
	&= \Bigl(\sum_\beta \langle D , T_i, T_j^* \rangle_\beta q^{\beta}\Bigr)_{i,j} \\
	&= \Bigl(\sum_\beta (D.\beta) \langle T_i, T_j^* \rangle_\beta q^{\beta}\Bigr)_{i,j} .
\end{align*}
Now $D_i.\beta$ is easy to compute, so the commutativity relation above gives quadratic equations for the invariants $\langle T_i, T_j^* \rangle_\beta$. These equations are used repeatedly in this paper, and are much more transparent than the full WDVV equations.

\begin{definition}\label{dNaiveAtomsMorePrecise}
The \emph{naive atomic decomposition} of $X$ is the ordered sequence of multiplicities of eigenvalues of $M_{(X, -K_X)}$ for an assignment of  generic values (in $\mathbb{C}$  here, say) to the parameters $q^{\beta}$. 

The \emph{small atomic decomposition} of $X$ is the ordered sequence of lengths of connected components of the Artinian scheme  $\mathrm{Spec}\, QH^{\bullet}_{\mathrm{alg}} (X)$, for generic values of the Novikov parameters.
\end{definition}

\begin{remark}\label{rNaive}
Hodge atoms $\alpha$ of $X$ as defined in \cite{KKPY25} are highly structured entities, and the information our naive and small atomic decompositions capture is much coarser. In particular, their atoms are labelled by the eigenvalues of an operator that is a generisation of $(-K_X)\star$, namely, multiplication by the Euler vector field using the big quantum cohomology. Compare in particular, the example of a complete intersection of type $(2,2)$ in $\mathbb{P}^6$ in \cite[Example 2.13, p. 28]{CKK25} where the naive atomic decomposition is $(9,1,1,1)$, but quantum multiplication by the Euler vector field has $12$ simple eigenvalues \cite{Hu21}. See also Section \ref{sCIQuadrics}. Still, naive atomic decompositions are much easier to compute with, and they also admit a variant for $G$-varieties, $G$ a finite group, where in addition to multiplicities one takes dimensions of $G$-invariants in generalised eigenspaces into account as well. 
\end{remark}

For a connection between the  eigenvalues of $M_{(X, -K_X)}$ and the coordinates of
singular fibres of associated Landau--Ginzburg models compare \cite{Pr18, Pr26}.

%Atoms $\alpha$ of $X$ as defined in \cite{KKPY25} are highly structured entities, but for our humble purposes we simply use that the atoms of $X$ are in natural bijection with, in other words, can be labelled by, the eigenvalues of $M_{(X, -K_X)}$ for an assignment of  generic values (in $\mathbb{C}$  here, say) to the parameters $q^{\beta}$. Moreover, one then has an important quantity $\rho_{\alpha}$ attached to an atom $\alpha$, namely the dimension of the corresponding generalised eigenspace (=the algebraic multiplicity of the eigenvalue labelling the atom). 

We will also use quantum periods and recursion relations derived from differential equations expressing the flatness of the quantum connection in Sections \ref{sCI23}, \ref{sRankOneFanoThreefolds}, \ref{sVerraFourfolds}. We will  explain that approach further in the relevant sections, and refer to \cite{CCGGK14} for general information about quantum periods and further references to the literature.

\section{Naive and small variants of Iritani's theorem}\label{sIritani}

In view of Iritani's blowup formula \cite{Iritani23} it is natural to define the following in our situation.  

%According to our discussion in the previous section, atom theory associates to each smooth projective Fano variety $X$ a finite unordered set of integers (possibly with repeats)
%\[
%X \rightsquigarrow \{ (\rho_{\alpha}, \alpha) %\} 
%\]
%labelled by the atoms $\alpha$ of $X$. Iritani's main theorem \cite{Iritani23} has the following astonishing consequence of a purely algebraic nature, and it is this version of his result that we are mainly going to illustrate and get a better felling for in the examples in subsequent sections. 

\begin{definition}\label{dIritaniMorePrecise}
Let $\widetilde{X}$ be the blowup of $X$ in a smooth codimension $c$ subvariety $Z$. 

\begin{enumerate}
\item
We say that \emph{the naive variant of Iritani's theorem holds for $(X, Z)$}  if the naive atomic decomposition of $\widetilde{X}$ is the sequence obtained by concatenating and sorting the naive atomic decompositions of $X$ and $(c-1)$ times the naive atomic decomposition of $Z$. 
\item 
We say that \emph{the small variant of Iritani's theorem holds for $(X, Z)$}  if the small atomic decomposition of $\widetilde{X}$ is the sequence obtained by concatenating and sorting the small atomic decompositions of $X$ and $(c-1)$ times the small  atomic decomposition of $Z$. 
\end{enumerate}
\end{definition}

Here it may be safer to assume $X$, $Z$ and $\widetilde{X}$ to be Fano to avoid convergence problems for quantum products, but we ignore these for the time being because they play no role in our examples and have been addressed in \cite{KKPY25} using nonarchimedean methods.

In subsequent sections we will see that, in the situation of Definition  \ref{dIritaniMorePrecise}, the characteristic polynomials of $(-K_X)\star$ and $(-K_Z)\star$ can be recovered from that of $(-K_{\widetilde{X}})\star$ by some specialisation procedure in examples, or more precisely, as limits that correspond to taking specific leading terms. So the characteristic polynomial on the blowup is some \emph{generisation} of the corresponding polynomials on $X$ and $Z$. We consider it as \emph{one of the basic miracles of Iritani's theorem} (that lacks any classical explanation to our knowledge and that we do not understand conceptually so far) that the multiplicities of the entries in the naive atomic decompositions of $X$ and $Z$ are preserved under this generisation. Indeed, a random deformation of a polynomial will of course destroy the multiplicities of its roots. 

\section{Smooth cubic fourfolds}\label{sCubicFourfolds}

Let $X$ be a smooth cubic fourfold. We choose the following homogeneous basis in the space $A^{\bullet}(X)_{\Q}$ of algebraic cycles modulo homological equivalence with rational coefficients: 
\[
X, H, H^2, S_1, \dots , S_t, L, pt %this labels the rows
\]
where $H$ is the hyperplane class, $L$ the class of a line, and $S_1, \dots , S_t$ classes of surfaces that together with $H^2$ form a basis of $A^{2}(X)_{\Q}$. The degree-wise dual basis is then
\[
pt, L, (H^2)^*, (S_1)^* , \dots , (S_t)^*, H, X . %this labels the columns
\]
Note that the monoid of effective algebraic curve classes modulo numerical equivalence, $\mathrm{NE}(X, \Z)$, is generated by the class of a line $L$. Let $\Lambda = \C [ \mathrm{NE}(X, \Z)]$ be the monoid ring generated by $q^{\beta}$ where $\beta = L$. 

\begin{theorem}\label{tMatrixCubic}
With the above notation and choice of basis, the matrix of quantum multiplication by $-K_X$ is 
\[
\begin{pmatrix}
0 & 3 & 0 & 0 & \dots & 0 & 0 & 0 \\
0 & 0 & 3 & 0 & \dots & 0 & 0 & 0 \\
6 \deg(H^2) q^L & 0 & 0 & 0 & \dots & 0 & 3\deg (H^2) & 0 \\
6 \deg(S_1) q^L & 0 & 0 & 0 & \dots & 0 & 3\deg (S_1) & 0 \\
\vdots & \vdots & \vdots & \vdots & \ddots & \vdots & \vdots & \vdots \\
6 \deg (S_t) q^L & 0 & 0 & 0 & \dots & 0 & 3\deg (S_t) & 0 \\
0 & 15q^L & 0 & 0 & \dots & 0 & 0 & 3 \\
0 & 0 & 6q^L & 0 & \dots & 0 & 0 & 0 
\end{pmatrix}.
\]
The characteristic polynomial of this matrix is
\[
\lambda^{2+t} (\lambda^3 - 729 q^L)
\]
and therefore the naive atomic decomposition of $X$ is
\[(2+t, 1,1,1).\]
\end{theorem}

\begin{proof}
Recall that the entries of the matrix we are about to compute are
\[
\langle\langle -K_X, T_i, T_j^* \rangle \rangle = \sum_{\beta \in \mathrm{NE}(X, \Z)} \langle -K_X, T_i, T_j^* \rangle_{\beta} q^\beta
\]
where $T_i$ and $T_j^*$ are the $i$-th and $j$-th elements of our basis and dual basis. The anticanonical is $-K_X =3H$, and the class $L$ has anticanonical degree $3$. Since for nonzero $\langle -K_X, T_i, T_j^* \rangle_{\beta}$ we need to have $1 + \mathrm{codim}\, T_i + \mathrm{codim}\, T_j^* = 4 + (-K_X).\beta$ and $\mathrm{codim}\, T_i + \mathrm{codim}\, T_j^*$ can be at most $8$, we see that for $\beta \neq 0$, we can only get a nonzero contribution from $\langle -K_X, T_i, T_j^* \rangle_{\beta}$  if $\beta = L$. This explains all zero entries in the matrix in the statement of the Theorem as well as all classical intersection number not involving $q^L$. Thus we only need to compute 
\[
\langle -K_X, \Sigma , pt  \rangle_{L } \quad \mathrm{and} \quad \langle -K_X, L , L \rangle_{L }
\]
where $\Sigma$ is a surface in $X$, to finish the proof. Now
\[
\langle -K_X, \Sigma , pt  \rangle_{L } = ((-K_X).L) \langle \Sigma , pt  \rangle_{L } = 3 \langle \Sigma , pt  \rangle_{L } 
\]
and $\langle \Sigma , pt  \rangle_{L }$ is equal to the intersection number of $\Sigma$ and the surface swept out by the lines passing through a general point on $X$. This latter surface has class $2H^2$: indeed, a line through a point on $X$ needs to be contained in the intersection of the tangent hyperplane to $X$ at that point and $X$, which for general choices is a nodal cubic threefold; and a line through the node of that cubic threefold is contained in the cubic threefold if and only if it is also contained in the intersection of that cubic threefold with its tangent quadratic cone at the point. This explains all the entries in the first column of the matrix in the statement of the Theorem. 

Finally, 
\[
\langle -K_X, L , L  \rangle_{L } = ((-K_X).L) \langle L , L  \rangle_{L } = 3 \langle L , L  \rangle_{L } 
\]
where $\langle L , L  \rangle_{L }$ is the number of lines on $X$ meeting two other general lines contained in $X$. This number is equal to five: indeed, two general lines contained in $X$ span a $\P^3$ whose intersection with $X$ is a smooth cubic surface; and the number of lines on a cubic surface meeting two other lines on the surface is classically known to be five. This finishes the proof. 
\end{proof}

We compare this with Beauville's computation of the small quantum cohomology ring for Fano complete intersections of large enough index:

\begin{theorem}[{{\cite{Beau97}}}]\label{tBeauville}
Let $X \subset \mathbb{P}^{n+r}$ be a smooth complete intersection of degree $(d_1, \ldots, d_r)$ and dimension $n \geq 2$, with $n \geq 2\sum_{i=1}^r (d_i - 1) - 1$. The quantum cohomology algebra $H^*(X, \mathbb{Q})$ is the algebra generated by $H$ and the primitive middle cohomology $H^n(X, \mathbb{Q})_o$, with the relations:
\[
H^{\star(n+1)} = \mu(X) H^{\star(n+1-k)}, \quad\quad
H \star \alpha = 0, \quad\quad
\alpha \star \beta = (\alpha \cdot \beta) \frac{1}{d} \big( H^{\star n} - \mu(X) H^{\star(n-k)} \big)
\]
for $\alpha, \beta \in H^n(X, \mathbb{Q})_o$. Here $k = n + r + 1 - \sum d_i$ is the Fano index,  $d = \prod d_i$ the degree of $X$, and $\mu(X) = d_1^{d_1} \dots d_r^{d_r}$.
\end{theorem}

Indeed, for cubic $4$-folds which have Fano index $3$ we obtain
\[
    H^{\star 5} = 3^3 H^{\star 2} %\quad \text{and} \quad H \star \alpha = 0
\]
and hence
\[
    \left(\frac{-K_X}{3}\right)^{\star 5} 
    = 3^3 \left(\frac{-K_X}{3}\right)^{\star 2} 
    \implies (-K_X)^{\star 5} = 3^6 (-K_X)^{\star 2}.
\]
This shows that the characteristic polynomial for the matrix representing $-K_X$ on $1, H, H^2, L, pt$ is
\[
    \lambda^2(\lambda^3 - 729 q^L).
\]
Restricting our attention to the algebraic classes, the relation $-K_X \star \alpha = 0$ for $\alpha$ in the primitive cohomology shows that this part contributes another factor
\[
    \lambda^t
\]
to the characteristic polynomial.

\section{Cubic fourfolds blown up in a plane}\label{sPlane}

Consider a cubic fourfold $X$ containing a plane $P$ and general with that property. We denote by $\widetilde{X} = \mathrm{Bl}_P \, X$ the blowup of $X$ in $P$. On $\widetilde{X}$ we choose the following basis of the algebraic cycles modulo homological equivalence:
\[
\widetilde{X}, H, H^2, P, L, pt, E, -R, F
\]
where $H, H^2, P, L, pt$ denote the pullbacks to $\widetilde{X}$ of the corresponding classes on $X$, $E$ is the exceptional divisor, $R$ the preimage in $E$ of a line in $P$, and $F$ the class of a fibre of $E \to P$. 

\begin{proposition}\label{pIntersectionMatrix}
The intersection numbers between members of the previous ordered basis are given in the following matrix: 
\[
\renewcommand{\arraystretch}{1.2}
\begin{array}{c|cccccc|ccc}
      & \widetilde{X} & H & H^2 & P & L & \mathrm{pt} & E & -R & F \\ \hline
\widetilde{X} & 0 & 0 & 0 & 0 & 0 & 1 & 0 & 0 & 0 \\
H             & 0 & 0 & 0 & 0 & 1 & 0 & 0 & 0 & 0 \\
H^2           & 0 & 0 & 3 & 1 & 0 & 0 & 0 & 0 & 0 \\
P             & 0 & 0 & 1 & 3 & 0 & 0 & 0 & 0 & 0 \\
L             & 0 & 1 & 0 & 0 & 0 & 0 & 0 & 0 & 0 \\
\mathrm{pt}   & 1 & 0 & 0 & 0 & 0 & 0 & 0 & 0 & 0 \\ \hline
E             & 0 & 0 & 0 & 0 & 0 & 0 & 0 & 0 & -1 \\
-R            & 0 & 0 & 0 & 0 & 0 & 0 & 0 & -1 & 0 \\
F             & 0 & 0 & 0 & 0 & 0 & 0 & -1 & 0 & 0
\end{array}
\]
Recall that if the intersection is not zero-dimensional,  the intersection number is zero by definition. 

Therefore the degree-wise dual basis is 
\[
pt, L, (H^2)^*, P^*, H, \widetilde{X}, -F, R, -E .
\]

The generators satisfy the following relations:
\begin{gather*}
H.P = L , \quad 
H^3  = 3L , \quad 
H.L  = pt , \quad
P^2  = 3 pt, \quad
H. pt = 0, \\
E^2  = -P , \quad
E.P  = 3F , \\
H.E  = R , \quad
H^2. E  = F. 
\end{gather*}
\end{proposition}

\begin{proof}
By \cite[Chapter 6, 6.7]{Ful98}, the intersection theory on the blowup can be described as follows: Consider the diagram
\[
\xymatrix{
E \ar@{^{(}->}[r]^j \ar[d]^{\pi_E} & \widetilde{X} \ar[d]^{\pi} \\
P \ar@{^{(}->}[r]^i & X .
}
\]
We have the rules
\begin{align*}
\pi^* (\alpha ) . \pi^* (\beta ) = \pi^* (\alpha . \beta) ,  & \quad  \alpha, \beta \in A^{\bullet}(X)_{\Q}, \\
j_* (\gamma ) . \pi^* (\alpha ) = j_* ( \gamma . \pi_E^* i^* (\alpha )) , & \quad \alpha \in A^{\bullet}(X)_{\Q}, \: \gamma \in A^{\bullet}(E)_{\Q}, \\
j_* (\gamma ) . j_* (\delta ) = - j_* (\gamma . \delta . \zeta) , & \quad \gamma , \delta \in A^{\bullet}(E)_{\Q}
\end{align*}
where $\zeta = c_1 (\mathcal{O}_{\P (N_{P/X})}(1))$. With this all the entries are obvious except $P^2$. For this recall 
\[
P^2 \;=\; \deg\bigl(c_2(N_{P/X})\bigr).
\]

Use the normal bundle exact sequence
\[
0 \longrightarrow N_{P/X} \longrightarrow N_{P/\mathbb P^5}\cong \mathcal O_P(1)^{\oplus 3}
\longrightarrow N_{X/\mathbb P^5}\big|_P \cong \mathcal O_P(3) \longrightarrow 0 .
\]
Hence
\[
c(N_{P/X})=\frac{c\!\left(\mathcal O_P(1)^{\oplus 3}\right)}{c\!\left(\mathcal O_P(3)\right)}
=\frac{(1+h)^3}{1+3h}
=1+3h^2,
\]
so
\[
c_2(N_{P/X})=3h^2 = 3.
\]
Here $h$ is the class of a line in $P$. 

\medskip

Among the relations, only 
\[
E^2  = -P , \quad 
E.P  = 3F 
\]
require proof. The first relation follows from $E^2 =-\zeta$ together with the blowup relation $P =\zeta + j_* (\pi_E^* c_1 (N_{P/X}))$ \cite[Thm. 13.14]{E-H16} and the fact that $c_1 (N_{P/X}) =0$. The second equation follows from $P^2 =3 pt$. 
\end{proof}

\begin{proposition}\label{pMonoidOfCurves}
We have
\[
\mathrm{NE}(\widetilde{X}, \Z ) = \langle F, L-F \rangle
\]
as monoids.
\end{proposition}

\begin{proof}
 Indeed, $F$ and $L-F$ are both effective and generate $A^{3}(\widetilde{X})_{\Q}$. Now $H$ and $H-E$ are nef because the former is the pullback of an ample class on $X$ to $\widetilde{X}$ whereas the latter defines a morphism to $\P^2$. Thus if $aL - bF$ is any class in $\mathrm{NE}(\widetilde{X}, \Z )$, which is not a multiple of $F$, then intersecting with $H$ and $H-E$ gives $a>0$ and $a\ge b$, whence the assertion. 
 \end{proof}

The anticanonical on $\widetilde{X}$ is $3H-E$. We will set $q=q^L$ and $t^{-1} = q^F$ where $t$ is another formal variable, for the elements $q^L$ and $q^F$ of $\C [ \mathrm{NE}(\widetilde{X}, \Z ) ]$. Here the exponent of $t$ occurring in $q^{\beta}$ can be interpreted as the contact order $E.\beta$ of $\beta$ with the exceptional divisor. 

Let $M(\widetilde{X})$ be the matrix given by quantum multiplication by $-K_{\widetilde{X}}$ in the above basis.

\begin{theorem}\label{t:FullMatrixBlowup}
The matrix $M(\widetilde{X})$ is 
\[
\left(\!\begin{array}{cccccc|ccc}
       0&3&0&0&0&0&-1&0&0\\
       4qt&0&3&0&0&0&0&1&0\\
       18q&14qt&0&0&9&0&-10qt&0&-1\\
       6q&-6qt&0&0&3&0&2qt&0&-3\\
       0&15q&6qt&-4qt&0&3&0&10qt&0\\
       20q^{2}t&0&6q&0&4qt&0&0&0&-4qt\\ \hline
       4qt&0&0&1&0&0&-\frac{1}{t}&-3&0\\
       0&-10qt&0&0&-1&0&6qt&-\frac{1}{t}&-3\\
       0&0&4qt&-2qt&0&1&3q&6qt&-\frac{1}{t}
       \end{array}\!\right).
\]
The characteristic polynomial of this matrix is
%\tiny
%\[
%\frac{\lambda^{3}\left(4096q^{3}t^{6}-768q^{2}t^{5}\lambda^{2}+48qt^{4}\lambda^{4}-t^{3}\lambda^{6}+13056q^{2}t^{4}\lambda+960qt^{3}\lambda^{3}-3t^{2}\lambda^{5}-3456q^{2}t^{3}+
%       2424qt^{2}\lambda^{2}-3t\lambda^{4}+2268qt\lambda-\lambda^{3}+
%       729q\right)}{t^{3}}.
%\]
\normalsize
\begin{gather*}
\frac{1}{t^{3}}\lambda^3\Bigl(\lambda^{6}(-t^{3})
+\lambda^{5}(-3t^{2})
+\lambda^{4}(48qt^{4}-3t)
+\lambda^{3}(960qt^{3}-1) 
+\lambda^{2}(-768q^{2}t^{5}+2424qt^{2}) \\
+\lambda(13056q^{2}t^{4}+2268qt)
+(4096q^{3}t^{6}-3456q^{2}t^{3}+729q)\Bigr)
\end{gather*}
and hence
\[
\lim_{t\to 0} \bigl( t^3 \chi_{\widetilde{X}}(\lambda) \bigr) = -\lambda^3(\lambda^3-729q)
\]
and
\[
\lim_{t\to 0} \bigl( t^6 \chi_{\widetilde{X}} (\lambda - t^{-1}) \bigr) = -(\lambda^3-27q).
\]
\end{theorem}

\begin{proof}
We use the following Gromov-Witten invariants as input:
\begin{align*}
\langle L , L \rangle_{L} &= 5  \quad \text{(the number of lines meeting two given ones on a cubic surface)} \\
	\langle H^2, pt \rangle_{L} &= 6  \quad \text{(the degree of the cone of lines through a point on $X$)} \\
	\langle P, pt \rangle_L &= 2  \quad \text{} \\
	\langle R, R \rangle_{F} &= 1  \quad \text{} \\
	\langle F\rangle_{F} &= -1 . \\
\end{align*}
We have $\langle P, pt \rangle_L = 2 $ because the surface of lines through a point on $X$ has class $2H^2$ and $P\cdot(2H^2) =2$.

By multilinearity, 
\[
 \langle F \rangle_F = - \langle L -F \rangle_F + \langle L \rangle_F
\]
and whereas $ \langle L \rangle_F =0$, we see that $\langle L -F \rangle_F =1$ since there is one fiber through the strict transform of a line intersecting $P$ in a point. So $\langle F\rangle_{F} = -1$.

Likewise, we claim 
\[
\langle R, R \rangle_{F} = 1 .
\]
Indeed, for this it suffices to show
\[
\langle  H^2-  R,  H^2 -R \rangle_{F} = 1,
\]
which is true because two general surfaces of class $H^2$ in $X$ that intersect $P$ in two different lines will have strict transforms on $\widetilde{X}$ that can be joined by precisely one curve of class $F$, the fiber over the intersection point of the lines on $P$.

Using  the computer algebra program Macaulay2, we find that the commutativity relation $E \star H = H \star E$ completely determines all the Gromov-Witten invariants contributing to $
M(\widetilde{X})$ and give the values claimed \cite{{BvBS26}}. 
The characteristic polynomial and limits are also computed in \cite{BvBS26}. %{\color{red} Do we want to give more precise references, here and in the following, to the part of the code/files? What about Zac's computations in later sections? Do we also want to upload that code to Zenodo?}
\end{proof}

\begin{remark}
Notice that $M(\widetilde{X})_{\le 0}$, the degree $\le 0$ part in $t$  of the matrix $M(\widetilde{X})$,  is equal to
\[
\renewcommand{\arraystretch}{1.2}
\begin{array}{c|cccccc|ccc}
      & \mathrm{pt} & L & (H^2)^* & P^* & H & \widetilde{X} & -F & R & -E  \\ \hline
\widetilde{X} & 0 & 3 & 0 & 0 & 0 & 0 & -1 & 0 & 0 \\
H             & 0 & 0 & 3 & 0 & 0 & 0 & 0 & 1 & 0 \\
H^2           & 18q & 0 & 0 & 0 & 9 & 0 & 0 & 0 & -1 \\
P             & 6q & 0 & 0 & 0 & 3 & 0 & 0 & 0 & -3 \\
L             & 0 & 15q & 0 & 0 & 0 & 3 & 0 & 0 & 0 \\
\mathrm{pt}   & 0 & 0 & 6q & 0 & 0 & 0 & 0 & 0 & 0 \\ \hline
E             & 0 & 0 & 0 & 1 & 0 & 0 & -t^{-1} & -3 & 0 \\
-R            & 0 & 0 & 0 & 0 & -1 & 0 & 0 & -t^{-1} & -3 \\
F             & 0 & 0 & 0 & 0 & 0 & 1 & 3q & 0 & -t^{-1}
\end{array}
\]
The upper left $6\times 6$ submatrix of $M(\widetilde{X})_{\le 0}$ is the matrix given by quantum multiplication by $-K_X$ according to Theorem \ref{tMatrixCubic}, and hence 
\[
\lim_{t\to 0} \bigl( t^3 \chi_{\widetilde{X}}(\lambda) \bigr) = -\chi_X (\lambda) .
\]
Moreover, the lower $3\times 3$ matrix of $M(\widetilde{X})_{\le 0} + t^{-1} \mathrm{Id}_{9\times 9}$ is, up to sign change in the basis element $R$, the matrix given by quantum multiplication by $-K_P$ on $P=\P^2$. Therefore, 
\[
\lim_{t\to 0} \bigl( t^6 \chi_{\widetilde{X}} (\lambda - t^{-1}) \bigr) = -\chi_P (\lambda ). 
\]
\end{remark}

%{\color{red} what does the differential equation give in this case?}

\section{$\P^4$ blown up in a complete intersection of multidegree $(2,2,2)$}\label{sP4BlownUp}

Let $C$ be a complete intersection curve of multidegree $(2,2,2)$ in $X=\P^4$, and let $\widetilde{X}$ be the blowup of $X$ in $C$. Let $E$ be the exceptional divisor, $H$ the pull-back of the hyperplane class, $L$ the pullback of the class of a line, $P$ the class of a fibre of $E$ over $C$, $S$ the class $c_1(\mathcal{O}_{\P (N_{C/X}}(1))$, and $F$ the class of a line in a fibre. 
\begin{lemma}\label{lXFano}
The variety $\widetilde{X}$ is Fano.
\end{lemma}

\begin{proof}
 We have $-K_{\widetilde{X}} = 5H - 2E$. Every effective curve class on $\widetilde{X}$ with support in a fibre of $E$ over $C$ is a positive multiple of $F$ and $-K_{\widetilde{X}}.F = 2$. Suppose $\Gamma$ is an irreducible curve on $\widetilde{X}$ not contained in a fibre of $E$. Then $\Gamma$ is numerically equivalent to the strict transform of a curve $\overline{\Gamma}$ on $X$. Suppose $\overline{\Gamma}.H =d$. Then $\overline{\Gamma}$ intersects $C$ in a subscheme of length at most $2d$. This means that  $-K_{\widetilde{X}}.\Gamma \ge  1$. Hence  $-K_{\widetilde{X}}$ is ample by Kleiman's criterion.
\end{proof}

\begin{proposition}\label{pRelationsINtersectionRing}
We have the following relations in the intersection ring of $\widetilde{X}$:
\begin{align*}
H^5 &= 0 \\
H^2 E &= 0 \\
(2H -E)^3 &= 0 \\
8P - HE &= 0 \\
S + E^2 &= 0 \\
F + PE &= 0 \\
L - H^3 &= 0 \\
pt - H^4 &= 0.
\end{align*}
\end{proposition}

\begin{proof}
The only relation that requires an explanation is $(2H -E)^3=0$, which holds because the quadrics through $C$ give a morphism of $\widetilde{X}$ to $\P^2$.
\end{proof}

Notice that the monoid of effective curve classes on $\widetilde{X}$ is generated by $F$ and $L-2F$. 

\begin{lemma}\label{lDualBasis}
The dual basis to the basis  $\{1,H,H^2,L,pt,E,P,S,F\}$ is given by
\[
\{ pt,L,H^2,H,1,-F,-S-48P,-P,-E\}. 
\]
\end{lemma}

\begin{proof}
This is a direct computation.
\end{proof}

\begin{theorem}\label{t222}
Let $\widetilde{X}$ be the blow up of $\bP^4$ in a complete intersection curve $C$ of type $(2,2,2)$. Then with $q= q^L$ and $t = (q^F)^{-1}$ as before we have that the matrix $M_{(\widetilde{X}, -K_{\widetilde{X}})}$ representing quantum multiplication by $-K_{\widetilde{X}}$ is given by 
{\tiny
\[
\left(\!\begin{array}{ccccc|cccc}
      0&5&0&0&0&-2&0&0&0\\
      40q^{2}t^{4}&16qt^{2}&5&0&0&-6qt^{2}&-16&0&0\\
      192q^{3}t^{6}+24qt&160q^{2}t^{4}&44qt^{2}&5&0&-60q^{2}t^{4}&-288qt^{2}&-7qt^{2}&0\\
      160q^{2}t^{3}&192q^{3}t^{6}+24qt&160q^{2}t^{4}&16qt^{2}&5&-96q^{3}t^{6}-3qt&-1120q^{2}t^{4}&-30q^{2}t^{4}&-32qt^{2}\\
      480q^{3}t^{5}+5q&160q^{2}t^{3}&192q^{3}t^{6}+24qt&40q^{2}t^{4}&0&-40q^{2}t^{3}&-1536q^{3}t^{6}-120qt&-48q^{3}t^{6}-3qt&-80q^{2}t^{4}\\ \hline
      80q^{2}t^{4}&32qt^{2}&0&0&0&-12qt^{2}&40&2&0\\
      48q^{3}t^{6}+3qt&30q^{2}t^{4}&7qt^{2}&0&0&-10q^{2}t^{4}&-44qt^{2}&-qt^{2}&2\\
      -768q^{3}t^{6}-24qt&-320q^{2}t^{4}&-48qt^{2}&16&0&\frac{80q^{2}t^{5}-2}{t}&256qt^{2}&4qt^{2}&-56\\
      40q^{2}t^{3}&96q^{3}t^{6}+3qt&60q^{2}t^{4}&6qt^{2}&2&-48q^{3}t^{6}&\frac{-400q^{2}t^{5}-2}{t}&-10q^{2}t^{4}&-12qt^{2}
      \end{array}\!\right)
\]
}
\normalsize
The characteristic polynomial of this matrix is
\[
\chi (\lambda ) = 
\frac{F_2^{2}\cdot F_5}{t^2}
\]
with
\[
F_2 = \left(16q^{2}t^{5}+8q\lambda t^{3}+\lambda
      ^{2}t+4\right) = -\left(4iq\sqrt{t}^5+i\lambda
       \sqrt{t}-2\right)\left(4iq\sqrt{t}^{5}+i\lambda
       \sqrt{t}+2\right)
       \]
\begin{align*}
F_5  = & 27648q^{5}t^{10}+6912q^{4}\lambda
      t^{8}-1152q^{3}\lambda ^{2}t^{6}+19584q^{3}t^{5}\\ 
      & -160q^{2}\lambda
      ^{3}t^{4}+4800q^{2}\lambda t^{3}+28q\lambda^{4}t^{2}+1000q\lambda
      ^{2}t-\lambda ^{5}+3125q.
\end{align*}
The coefficient of the leading term of $\chi$ with respect to $t$ is 
$$4^2(\lambda ^{5}+3125q),$$ 
which recovers the characteristic polynomial of $\P^4$.
Shifting 
$$\lambda \mapsto \lambda - \frac{2i}{\sqrt{t}}$$
the coefficient of the  leading term with respect to $\sqrt{t}$ is 
$$\lambda^2\cdot(2i)^2\cdot 32$$
recovering the characteristic polynomial of the blow up center $C$.
Shifting 
$$\lambda \mapsto \lambda + \frac{2i}{\sqrt{t}}$$
the coefficient of the  leading term with respect to $\sqrt{t}$ is 
$$\lambda^2\cdot(2)^2\cdot 32$$
again recovering the characteristic polynomial of the blow up center $C$.
\end{theorem}

\begin{proof}
This follows from a computation \cite{BvBS26}, using the commutativity relation for the quantum product by divisors and the following numbers as input: 

\begin{align*}
	\langle L \rangle_{L-2F} &= 16  \quad \text{(the number of nodes after projection from a line)} \\
	\langle pt \rangle_{2L-4F} &= 10  \quad \text{(the number of 4-secants after projecting from a point)} \\
	\langle pt,pt \rangle_L &= 1  \quad \text{(one line through 2 points)} \\
	\langle H^2,pt\rangle_{L-F} &= 8  \quad \text{(degree of the cone over $C$ with vertex a point outside $C$)} \\
	\langle F\rangle_{L-2F} &= 6 . \\
\end{align*}	
The first number counts the number of $2$-secants to $C$ that intersect a general line in $\P^4$. We project from this line. The $2$-secants that intersect the line correspond to the nodes of the projected curve in $\P^2$. Since $C$ has degree $8$ and genus $5$ the projection has ${8-1 \choose 2}-5 = 16$ nodes.

The number $\langle pt \rangle_{2L-4F}$ is equal to the number of $4$-secant conics to $C$ passing through a given general point $P$ outside of $C$. Projecting from $P$ maps $4$-secant conics of $C$ to $4$-secant lines of the image $C'\subset \P^3$ of $C$. Conversely the preimage of a $4$-secant line to $C'$ is a $\P^2$ containing $P$ and $4$ points of $C$. This $\P^2$ contains at least one conic passing through all $5$ points. By Bezout the $5$ points can only be contained in $2$ conics, if at least $4$ of the $5$ points are on a line. This line would be at least a $3$ secant to $C$. But since $C$ is cut out by quadrics, it cannot have any $3$-secant lines. So the number of $4$-secant conics to $C$ is the same as the number of $4$-secant lines to $C'$. These are counted by Cayley's $4$-secant formula
\[
	N = \frac{(d-2)(d-3)^2(d-4)}{12} - \frac{g(d^2-7d+13-g)}{2}.
\]
Plugging in $d=8$ and $g=5$ gives $N=10$. 

The third equation is obvious.

The number $\langle H^2,pt\rangle_{L-F}$ is the number of lines passing through a given general point and intersecting $C$ and a general plane. The lines passing through that point and intersecting $C$ trace out the cone over $C$, and its degree is $8$ since the degree of $C$ is $8$. 

The number $\langle F\rangle_{L-2F}$ counts the number of two secant lines to $C$ passing through a fixed general point $p$ on $C$ with normal direction constrained to lie in a given general hyperplane of normal directions. Let $H$ be the hyperplane of $\P^4$ spanned by this. $H$ contains the tangent line to $C$ at $p$. Therefore $H\cap C$ consists of $8$ points one of which is a double point at $p$. The lines connecting $p$ to the residual six points are exactly the $2$-secants through $p$ with the appropriate normal direction.

\end{proof}

\section{Cubic fourfolds blown up in a line}\label{sLine}

Let $X$ be a cubic fourfold, and let $\widetilde{X}$ be the blow up of $X$
along a line $L \subset X$. We denote by $E$ the exceptional divisor. Let
$P$ be the class of a fibre of the projection $E \to L$, and let $F$ be the
class of a line in such a fibre. Finally, we set
\[
S := -E^2,
\]
which is the class of a relative line.

\begin{proposition}\label{pIntersectionRing}
The intersection ring of $\widetilde{X}$ is generated, as a $\Z$-module,
by the classes
\[
    1,H,H^2,L,pt,E,P,S,F
\]
and its product structure is determined by the relations
\begin{gather*}
    H^3 = 3L, \qquad
    \mathrm{pt} = H L, \qquad
    P = E H, \qquad
    S = -E^2, \qquad
    F = S H, \qquad
    H^2E = 0, \\
    (H-E)^3 = 2(L-F).
\end{gather*}
\end{proposition}

\begin{proof}
The first five relations are definitions. Moreover, $H^2E=0$, since $H^2$
is represented by a codimension-two linear section which does not meet the
blowup center $L$.

It remains to prove the last relation. Consider the projection from the line
$L$ to $\P^3$. After blowing up $L$, this projection is given by the
linear system $|H-E|$ on $\widetilde{X}$. Hence the class of a fibre is $(H-E)^3$.

On the other hand, such a fibre is the transform of the residual curve to
$L$ in the intersection of $X$ with a plane containing $L$. This residual
curve is a conic, since $X$ is cubic, and it meets $L$ in two points.
Therefore its strict transform has class $2L - 2F$. The last relation follows. 
\end{proof}

Notice that the monoid of effective curve classes on $\widetilde{X}$ is generated by $F$ and $L-F$. 

\begin{lemma}\label{lDualBasis}
The dual basis to the basis  $\{1,H,H^2,L,pt,E,P,S,F\}$ is given by
\[
\{ pt,L,\frac{1}{3}H^2,H,1,-F,-S-P,-P,-E\}. 
\]
\end{lemma}

\begin{proof}
This is a direct computation.
\end{proof}

\begin{theorem}\label{tMatrixCubicFourfoldLine}
Let $\widetilde{X}$ be the blowup of $X$ along a line $L$. As before, set $q = q^L$ and $t = (q^F)^{-1}$.
Then the matrix $M_{(\widetilde{X},-K_{\widetilde{X}})}$ representing quantum
multiplication by $-K_{\widetilde{X}}$ is given by
\[
\renewcommand{\arraystretch}{1.2}
\begin{array}{c|ccccc|cccc}
	& pt & L & \frac{1}{3}H^2 & H & \widetilde{X} & -F & -S-P & -P & -E \\
	\hline
       \widetilde{X} & 0&3&0&0&0&-2&0&0&0\\
       H & 4q^{2}t^{2}&5qt&3&0&0&-5qt&-2&0&0\\
       H^2 & 18q&6q^{2}t^{2}&7qt&9&0&-6q^{2}t^{2}&-15qt&-6qt&0\\
       L & 56q^{2}t&15q&2q^{2}t^{2}&5qt&3&0&-4q^{2}t^{2}&-2q^{2}t^{2}&-5qt\\
       pt & 80q^{3}t^{2}&56q^{2}t&6q&4q^{2}t^{2}&0&-24q^{2}t&0&0&-4q^{2}t^{2}\\
       \hline
       E & 4q^{2}t^{2}&5qt&0&0&0&-5qt&3&2&0\\
       P & 0&2q^{2}t^{2}&2qt&0&0&-2q^{2}t^{2}&-5qt&-qt&2\\
       S & 0&2q^{2}t^{2}&3qt&2&0&\frac{-2q^{2}t^{3}-2}{t}&-5qt&-4qt&1\\
       F & 24q^{2}t&0&2q^{2}t^{2}&5qt&2&3q&\frac{-4q^{2}t^{3}-2}{t}&-2q^{2}t^{2}&-5qt
 \end{array}
\]
The characteristic polynomial is
\[
    \chi =
    \frac{(2qt+\lambda)^2 F_7}{t^2},
\]
where $F_7$ is an irreducible polynomial of degree $7$. 

The coefficient of the leading term of $\chi$ with respect to $t$ is
\[
    -16\lambda^2(\lambda^3 - 729q),
\]
which recovers the characteristic polynomial of the cubic fourfold $X$.

After the change of variables
\[
    \lambda \mapsto \lambda + \frac{2i}{\sqrt{t}},
    \qquad
    q \mapsto \frac{q}{i\sqrt{t}},
\]
the coefficient of the leading term with respect to $\sqrt{t}$ is
\[
    512(\lambda^2 - 4q),
\]
which recovers the characteristic polynomial of the blow-up center $L$.

Similarly, after the change of variables
\[
    \lambda \mapsto \lambda - \frac{2i}{\sqrt{t}},
    \qquad
    q \mapsto \frac{q}{i\sqrt{t}},
\]
the coefficient of the leading term with respect to $\sqrt{t}$ is
\[
    -512(\lambda^2 + 4q),
\]
which recovers the characteristic polynomial of the blow-up center $L$, up to
the sign change $q \mapsto -q$.
\end{theorem}

\begin{proof}
Apart from the commutativity relation $H \star E = E \star H$, the following
Gromov--Witten invariants were used as input for the computation of
$M_{(\widetilde{X},-K_{\widetilde{X}})}$ in Macaulay2:
\begin{align*}
    \langle L \rangle_{L-F}
        &= 5
        && \text{(lines on a cubic surface meeting two given lines)}, \\
    \langle \mathrm{pt} \rangle_{2L-2F}
        &= 1
        && \text{(planes through $L$ containing a general point)}, \\
    \langle H^2,\mathrm{pt}\rangle_L
        &= 6
        && \text{(degree of the cone swept out by the lines through a general point)}.
\end{align*}

The characteristic polynomial, the changes of variables, and the leading
terms in the corresponding limits were also computed using Macaulay2 \cite{BvBS26}.

Finally, the absolute irreducibility of the polynomial $F_7$ was checked using SageMath, {\ttfamily https://sagecell.sagemath.org}.
\end{proof}

\begin{remark}
Since $F_7$ is irreducible, the seven atoms of multiplicity one cannot be
distinguished globally. In the limiting regimes considered above, three of
them specialize to the atoms coming from the cubic, while the remaining
$2 \cdot 2 = 4$ specialize to the atoms coming from the line. Globally, however,
they are exchanged by monodromy.
\end{remark}

\begin{remark}
Observe that the shift of $\lambda$ is the same as for the blowup of $\P^4$
along a $(2,2,2)$ complete-intersection curve, which is also an example with
a blowup center of codimension $3$.
\end{remark}

%{\color{red} maybe this can be removed, or be used to check the limiting assertions/conjectures?} 

\section{Cubic threefold blown up in a plane elliptic curve}\label{sElliptic}

Let $X$ be a smooth cubic threefold. We start by computing the matrix given by quantum multiplication by $-K_X$ with respect to the basis of algebraic cycles 
\[
X, H, L, pt
\]
on $X$, where $H$ is a hyperplane section and $L$ the class of a line. The dual basis is 
\[
pt, L, H, X .
\]
We write $q=q^L$ as before. 

\begin{proposition}\label{p:CubicThreefold}
The matrix for quantum multiplication by $-K_X$ in the chosen basis above is 
\[
\begin{pmatrix}
0 & 2 & 0 & 0 \\
12q & 0 & 6 & 0 \\
0 & 10q & 0 & 2\\
24q^2 & 0 & 12q & 0
\end{pmatrix}
\]
and the characteristic polynomial of this is 
\[
\lambda^2 (\lambda^2 -108q)
\]
having one eigenvalue of multiplicity $2$ and two simple eigenvalues.
\end{proposition}

\begin{proof}
 Since in our present case $-K_X = 2H$, we get that the class $3L$ has anticanonical degree $6$, and thus in $\langle -K_X, T_i, T_j^* \rangle_{\beta}$ one of $T_i$ or $T_j^*$ would have to have codimension $4$ for a nonzero contribution for $\beta=3L$, impossible. For $\beta=2L$, we can only have a nontrivial contribution if $T_i$ and $T_j^*$ are both points, and then 
\[
\langle -K_X , pt, pt \rangle_{2L} = 4 \langle pt, pt \rangle_{2L}
\]
where $\langle pt, pt \rangle_{2L}$ is the number of conics on $X$ passing through two general points. Here $\langle pt, pt \rangle_{2L}$ is equal to six: a conic through general $p, q\in X$ lies in a plane containing the line through $p, q$, which in general intersects $X$ in a third point $r$. The line residual to the conic cut out by that plane must pass through $r$, and any line through $r$ defines such a conic in general by intersecting with the span of that line and the line through $p,q,r$. The number of lines through a general point on a cubic threefold is six. 

For $\beta=L$, we can only have nonzero contributions by $\langle -K_X , T_i, T_j^* \rangle_{L}$ if one of $T_i, T_j^*$ is a divisor and the other is a point, or if both are lines. In the first case, the expression becomes equal to $(-K_X.L)=2$ times the number of lines passing through a general point and intersecting the divisor (for the divisor $H$ we get $6$ for the number of such lines), and in the second case the expression becomes equal to $(-K_X.L)=2$ times the number of lines intersecting two other general lines on the cubic threefold (this is five because a line on a cubic surface is met by $10$ other lines coming in five pairs; the pairs lying in distinct planes).  

\end{proof}

Now consider a smooth plane cubic curve $C\subset X$, and the blowup $\widetilde{X}$ of $X$ in $C$. We fix a basis of the algebraic classes modulo homological equivalence on the blowup by choosing
\[
\widetilde{X}, H, L, pt, E, F 
\]
where we denote pullbacks of classes from $X$ by the same letters. Here $E$ is the exceptional divisor and $F$ the class of a fibre of $E$ over $C$.

\begin{proposition}\label{pIntersectionMatrixCubicThreefold}
The intersection numbers between members of the chosen ordered basis are given in the matrix below. 
\[
\renewcommand{\arraystretch}{1.2}
\begin{array}{c|cccc|cc}
      & \widetilde{X} & H & L & \mathrm{pt} & E & F \\ \hline
\widetilde{X} & 0 & 0 & 0 & 1 & 0 & 0 \\
H             & 0 & 0 & 1 & 0 & 0 & 0 \\
L             & 0 & 1 & 0 & 0 & 0 & 0 \\
\mathrm{pt}   & 1 & 0 & 0 & 0 & 0 & 0 \\ \hline
E             & 0 & 0 & 0 & 0 & 0 & -1 \\
F             & 0 & 0 & 0 & 0 & -1 & 0 
\end{array}
\]

Therefore, the degree-wise dual basis is:
\[
\mathrm{pt}, L, H, \widetilde{X}, -F, -E.
\]

The generators satisfy the following relations:
\begin{gather*}
H^2 = 3L, \quad
H \cdot L = \mathrm{pt}, \quad
H^3 = 3\mathrm{pt}, \quad
H \cdot \mathrm{pt} = 0, \\
E \cdot F = -\mathrm{pt}, \quad
E^2 = -3L + 6F, \\
H \cdot E = 3F, \quad
L \cdot E = 0, \quad
E^3 = -6\mathrm{pt}, \quad
F. H =0. 
\end{gather*}
\end{proposition}

\begin{proof}
The only nontrivial entry in the intersection matrix is $E.F =-1$, but this follows as in the proof of Proposition \ref{pIntersectionMatrix} since, with the obvious changes of notation, $E.F$ is equal to minus the intersection of $\zeta$ with $F$ computed on $E$. 

The dual basis is then clear. 

Among the relations the only ones whose proof requires an explanation are the ones for $E^3$ and $E^2$. We have that $E^3$ is equal to $\zeta^2$ computed on $E$ and by the projective bundle formula, this is equal to minus the first Chern class of the normal bundle of $C$ in $X$. 
Using the normal bundle exact sequence
\[
0 \longrightarrow N_{C/X} \longrightarrow N_{C/\mathbb P^4}\cong \mathcal O_C(1)^{\oplus 2} \oplus \mathcal{O}_C (3)
\longrightarrow N_{X/\mathbb P^4}\big|_C \cong \mathcal O_C(3) \longrightarrow 0 .
\]
we see that degree of the first Chern class of the normal bundle of $C$ in $X$ is $6$.

Let $\alpha, \beta$ be defined by $E^2 = \alpha L + \beta F$. Then 
\[
- 6 =E^3 = - \beta, 
\]
and
\[
E^2.H = E. (3F) = -3
\]
and hence
\[
-3 = E^2.H = \alpha L .H + \beta F.H = \alpha .
\]
\end{proof}

\begin{theorem}\label{tBlowupInEllipticCurve}
The matrix $M(\widetilde{X})$ is given by 
\[
\begin{array}{c|cccc|cc}
      & \mathrm{pt} & L & H & \widetilde{X} & -F & -E \\ 
      \hline
       \widetilde{X} &0&2&0&0&-1&0\\
      H &108q^{2}t^{2}+12q&15qt&6&0&-6qt&-3\\
      L & 252q^{3}t^{3}+126q^{2}t&84q^{2}t^{2}+10q&15qt&2&-48q^{2}t^{2}&-15qt\\
      pt & 1008q^{3}t^{2}+24q^{2}&252q^{3}t^{3}+126q^{2}t&108q^{2}t^{2}+12q&0&-252q^{3}t^{3}-18q^{2}t&-108q^{2}t^{2}\\ 
      \hline
      E &108q^{2}t^{2}&15qt&3&0& -6qt-\frac{1}{t}  &0\\
      F & 252q^{3}t^{3}+18q^{2}t&48q^{2}t^{2}&6qt&1&-12q^{2}t^{2}& -6qt-\frac{1}{t}  
      \end{array}
\]
and the characteristic polynomial is
\[
\chi = t^{-2}
\left(6qt+\lambda\right)^{2}
\left(6qt^{2}+t\lambda+1\right)^{2}
\left(441q^{2}t^{2}-42qt\lambda+\lambda^{2}-108q\right)
\]
Hence the leading term of $\chi$ with respect to $t$ is
\[
\left(\lambda\right)^{2}
\left(\lambda^{2}-108q\right)
\]
recovering the characteristic polynomial from Proposition \ref{p:CubicThreefold}. Also after shifting $\lambda \mapsto \lambda - \frac{1}{t}$. the leading term becomes
\[
\lambda^2
\]
which is what we expect for the blowup center.

\end{theorem}

\begin{proof}
We use the WDVV equations and the input
\begin{align*}
&\langle F \rangle_F = -1 \\
&\langle F \rangle_{L-F} = 6 \\
&\langle L \rangle_{L-F} = 15\\
&\langle F,F \rangle_{2L-2F} = 6.
\end{align*}
We justify the above equations as follows: $\langle L-F \rangle_F =1$ since for a general line meeting $C$ in one point, there is a unique fibre meeting its strict transform in one point. Thus by multilinearity, $\langle F \rangle_F = -1$.

\noindent
We have $\langle F \rangle_{L-F} = 6 $ because there are $6$ lines on $X$ passing through a general point on $C$.

\noindent
Then $\langle L \rangle_{L-F} = 15$ since given a general line $L$ on $X$, there are precisely $15$ lines on $X$ meeting $L$ and the elliptic curve $C$: projecting from $L$ to the plane spanned by the elliptic curve, we realise $X$ as a conic bundle over $\P^2$ with quintic discriminant curve that intersects $C$ in $15$ points. The fibre over each such point is a cross of lines, and in general only one of those lines will pass through the intersection point of the discriminant with $C$. 

\noindent
Given two general points $x,y$ on $C$ there are $6$ conics on $X$ passing through $x$ and $y$ because they correspond to lines through the residual intersection point $r$ with $X$ of the line $L_{xy}$ joining $x$ and $y$: indeed, each line $L_r$ through $r$ together with $L_{xy}$ spans a plane that intersects $X$ in $L_r$ and a conic through $x$ and $y$, and any conic through $x$ and $y$ lies in a plane that in general will intersect $X$ in a residual line $L_r$.

From these inputs the matrix, the characteristic polynomial and the leading terms are computed in \cite{BvBS26}.
\end{proof}

\section{Complete intersection of type $(2,3)$ in $\bP^5$}\label{sCI23}

Let $X= Q \cap F_3$ be a smooth complete intersection of a quadric $Q$ and a cubic $F_3$ in $\P^5$. As was proven in \cite[Prop. 4.2]{KS25}, if $Q$ is smooth, there are two semiorthogonal decompositions of the derived category $\mathrm{D}^b (X)$
\[
\mathrm{D}^b (X) = \big\langle  \mathrm{Ku}(X), \mathcal{O}_X, \Sigma^{\pm}_X \big\rangle 
\]
where $\Sigma^+_X, \Sigma^-_X$ are the restrictions of the two spinor bundles $\Sigma^+, \Sigma^-$ on $Q$ to $X$. It is not known but expected that these two decompositions are not in the same mutation orbit. M. Kontsevich has conjectured that there are, up to mutation equivalence, unique canonical semiorthogonal decompositions $\langle A_1,\dots,A_k \rangle $ of $\mathrm{D}^b(X)$ where $k$ is the number of distinct eigenvalues of big quantum multiplication by the Euler vector field \cite[Conjecture 1]{ESS25}. Because of the ambiguity introduced by the choice of spinor bundles, the present example is an interesting test case for this, although our specialisation to multiplication by $-K_X$ generally has fewer eigenvalues. As we will see below, there are just two components of the naive atomic decomposition in this case, and it is reasonable to expect that they correspond to $\langle \mathrm{Ku}(X), \overline{\Sigma}_X^{\pm}\rangle$ and to $ \mathcal{O}_X $ where $\overline{\Sigma}_X^{\pm}$ is the left mutation of the respective spinor bundle across $\mathcal{O}_X$. Thus the ambiguity here gets resolved by considering the category generated by the Kuznetsov component and the mutated spinor bundle as one atom. This agrees with the fact that for $Q$ a singular quadric, one only gets a decomposition
\[
\mathrm{D}^b (X) = \big\langle  \mathrm{Ku}(X), \mathcal{O}_X \big\rangle 
\]
and it is not known if the complement to $\mathcal{O}_X$ decomposes further in this case. 

We choose the basis of algebraic cycles
\[X,H,L,pt\] 
and its dual basis
\[pt, H, L, X.\]
The monoid of effective curves is generated by $L$, and we set $q = q^L$.
\begin{proposition}
    The matrix $M(X)$ is given by 
    \[
    \begin{pmatrix}
        0 & 1 & 0 & 0\\
        792q^2 & 30q & 6 & 0\\
        7272q^3 & 390q^2 & 30q & 1\\
        132192q^4 & 7272q^3 & 792q^2 & 0
    \end{pmatrix}
    \]
    and the characteristic polynomial is
    \[
    (\lambda-96q)(\lambda+12q)^3.
    \]
\end{proposition}
We note that this matrix was also written down in \cite{Pr06}.
\begin{proof}

    We use the quantum period \(G_X(t)\) to compute the matrix \(M(X)\), following
\cite{Fay26}. By Givental's mirror theorem \cite{Giv96}, the known
hypergeometric \(I\)-function of \(X\) determines the small \(J\)-function, and
\(G_X(t)\) is the component of this \(J\)-function in the direction of the unit
class, after restriction to the anticanonical slice. Comparing the associated scalar ODE
with this period determines the unknown coefficients of \(M(X)\), so no
individual non-classical three-point Gromov--Witten invariants are needed as
separate input.

    First, we compute an ansatz for $M(X)$ via constraints forced by the selection rule and classical intersection information, and symmetry of the basis.
    \[
    M(X) = 
    \begin{pmatrix}
        0 & 1 & 0 & 0 \\
        bq^2 & aq & 6 & 0 \\
        dq^3 & cq^2 & aq & 1 \\
        eq^4 & dq^4 & bq^2 & 0
    \end{pmatrix}
    \]
    with values $a,b,c,d,e$ to be found. Let $D = q \frac{q}{dq}$. From the first-order linear ODE
    \[
    Dy  = M(X)\cdot  y
    \]
    with solution 
    \[
     y = y_01 +y_1H + y_2 L + y_3pt \in H^*(X) \otimes \Lambda
    \]
    we obtain, through elimination of $y_1,y_2$ and $y_3$, the fourth-order scalar ODE satisfied by $f =y_0$:
    \[
    \begin{split}
    0 = D^4f -2aqD^3f & + (a^2q^2 -3aq -2bq^2-6cq^2)D^2f\\ 
    & + (2a^2q^2 + 2abq^3 -aq -4b^2 -12cq^2-12dq^3)Df \\
    & + (3abq^3 + b^2q^4 -4bq^2 -18dq^3 -6eq^4)f.
    \end{split}
    \] 
    The quantum period 
    \[
    G_X(t) = e^{-12t} \sum_{d=0}^{\infty} t^d \frac{(2d!)(3d!)}{(d!)^6}
    \]
    computed in \cite{CCGK16} with Givental's theorem is known to satisfy the above scalar ODE \cite{CCGGK14} and so by substituting
    \[
    \begin{cases}
        f\mapsto G_X(q)\\ q\mapsto t^{-K_X\cdot L} = t
    \end{cases}\]
    and comparing coefficients, we compute 
    \[
    a = 30,\quad  b = 792, \quad  c= 390, \quad  d= 7272,  \quad e = 132192.
    \]
    
\end{proof}
\section{Picard-rank-one Fano threefolds}
\label{sRankOneFanoThreefolds}

In this section we record the characteristic polynomial of $M(X)$, and the naive atomic decomposition for all Fano threefolds $X$ of Picard rank one, using the quantum periods for Fano threefolds computed in \cite{CCGK16}. The method is identical to that used in the preceding section. The full matrices and their characteristic polynomials for index 1 and 2 are computed in \cite{BvBS26}. These matrices were predicted in \cite{Gol07} and also computed later in \cite{Pr06} \cite{Pr07a} \cite{Pr07b}.
We use the notation
\[
  Q^3,\quad B_1,\ldots,B_5,\quad
  V_2,V_4,V_6,V_8,V_{10},V_{12},V_{14},V_{16},V_{18},V_{22}.
\]
Here $B_j$ denotes Fano threefold of index $2$ with
$H^3=j$, $V_k$ denotes the
Fano threefold of index $1$ and degree $k=(-K)^3$, and $Q^3$ denotes a cubic threefold. The index of each variety $X$ is denoted $i_X$, and the naive atomic decomposition is denoted $\mathrm{NAD}(X)$. We compare the resulting naive atomic decompositions with the standard
semiorthogonal decompositions collected in \cite{PS23}. In every case, the number of entries in the naive atomic decomposition agrees with the number of components in the displayed semiorthogonal decomposition. If Kontsevich's conjecture holds and the semiorthogonal decompositions provided are indeed the canonical ones, this suggests that the eigenvalue multiplicities do not split further when considering quantum multiplication by the Euler vector field in the big quantum cohomology.

\begin{table}[htbp]
\centering
\renewcommand{\arraystretch}{1.35}
\setlength{\tabcolsep}{3.5pt}

\begin{tabularx}{\textwidth}{@{} c c c >{\centering\arraybackslash}p{0.31\textwidth} X @{}}
\toprule
$X$ & $i_X$ & $\chi_X(\lambda)$ & NAD($X$) & SOD of $\Db(X)$ \\
\midrule

$\mathbb{P}^3$ & $4$ &
$\lambda^4 - 256q$ & $(1,1,1,1)$ & 
$\left\langle \OO,\OO(1),\OO(2),\OO(3) \right\rangle$ \\

\midrule

$Q^3$ & $3$ &
$\lambda(\lambda^3 - 108q)$ & ($1,1, 1,1)$ &
$\left\langle \mathcal S,\OO,\OO(1), \OO(2) \right\rangle$ \\

\midrule

$B_1$ & $2$ &
$\lambda^2(\lambda^2 - 1728q)$ & $(2,1,1)$ & 
$\left\langle \Ku(B_1),\OO,\OO(1) \right\rangle$ \\

$B_2$ & $2$ &
$\lambda^2(\lambda^2 - 256q)$ & $(2,1,1)$ & 
$\left\langle \Ku(B_2),\OO,\OO(1) \right\rangle$ \\

$B_3$ & $2$ &
$\lambda^2(\lambda^2 - 108q)$ & $(2,1,1)$ &
$\left\langle \Ku(B_3),\OO,\OO(1) \right\rangle$ \\

$B_4$ & $2$ &
$\lambda^2(\lambda^2 - 64q)$ & $(2,1,1)$ & 
$\left\langle \Db(C_2),\OO,\OO(1) \right\rangle$ \\

$B_5$ & $2$ &
$\lambda^4 - 44\lambda^2 q - 16q^2$ & $(1,1,1,1)$ &
$\left\langle \mathcal F_3,\mathcal F_2,\OO,\OO(1) \right\rangle$ \\

\midrule

$V_2$ & $1$ &
$(\lambda - 1608q)(\lambda + 120q)^3$ & $(3,1)$ & 
$\left\langle \Ku(V_2),\OO \right\rangle$ \\

$V_4$ & $1$ &  
$(\lambda - 232q)(\lambda + 24q)^3$ & $(3,1)$ & 
$\left\langle \Ku(V_4),\OO \right\rangle$ \\

$V_6$ & $1$  &
$(\lambda - 96q)(\lambda + 12q)^3$ & $(3,1)$ & 
$\left\langle \Ku(V_6),\OO \right\rangle$ \\

$V_8$ & $1$ &
$(\lambda - 56q)(\lambda + 8q)^3$ & $(3,1)$ & 
$\left\langle \Ku(V_8),\OO \right\rangle$ \\

$V_{10}$ & $1$ &
$(\lambda + 6q)^2(\lambda^2 - 32\lambda q - 244q^2)$ & $(2,1,1)$ & 
$\left\langle \Ku(V_{10}),\mathcal E_2,\OO \right\rangle$ \\

$V_{12}$ & $1$ &
$(\lambda + 5q)^2(\lambda^2 - 24\lambda q - 144q^2)$ & $(2,1,1)$ &
$\left\langle \Db(C_7),\mathcal E_5,\OO \right\rangle$ \\

$V_{14}$ & $1$ &
$(\lambda - 23q)(\lambda + 4q)^2(\lambda + 5q)$ & $(2,1,1)$ &
$\left\langle \Ku(V_{14}),\mathcal E_2,\OO \right\rangle$ \\

$V_{16}$ & $1$ &
$(\lambda + 4q)^2(\lambda^2 - 16\lambda q - 64q^2)$ & $(2,1,1)$ &
$\left\langle \Db(C_3),\mathcal E_3,\OO \right\rangle$ \\

$V_{18}$ & $1$ &
$(\lambda + 3q)^2(\lambda^2 - 12\lambda q - 72q^2)$ & $(2,1,1)$ &
$\left\langle \Db(C_2),\mathcal E_2,\OO \right\rangle$ \\

$V_{22}$ & $1$ &
$(\lambda + 4q)(\lambda^3 - 8\lambda^2 q - 56\lambda q^2 - 76q^3)$ & $(1,1,1,1)$ & 
$\left\langle \mathcal E_4,\mathcal E_3,\mathcal E_2,\OO \right\rangle$  \\
\bottomrule
\end{tabularx}
\label{tab:fano-threefold-charpoly-sod}
\end{table}
\FloatBarrier
\section{Verra fourfolds}\label{sVerraFourfolds}

In this section we compute the matrix $M(X)$ for the very general Verra fourfold $X$. We demonstrate that in Picard rank 2, the commutativity rule, selection rule, and classical information are not enough to determine the entire matrix of quantum multiplication by $-K_X$ without information from the geometry of $X$. The end result agrees with what is obtained in \cite{BGMP26}, but the present computation is different and arguably simpler.

Following the notation of \cite{Fay26}, we have the monomial basis of the algebraic cohomology
\[
\mathcal{B} = \{1, H_1, H_2, H_1^2, H_1H_2, H_2^2, H_1^2H_2, H_1H_2^2,H_1^2H_2^2\}
\]
The monoid of effective curves is generated by classes $L_1, L_2$ that satisfy $H_iL_j = \delta_{ij}$. We set $q_i = q^{L_i}$.
We compute the most general form of matrices $M_{(X,H_1)}$ and $M_{(X,H_2)}$ using only classical cup-product information and the selection rule. 
\[
M_{(X,H_1)}
=
\begin{pmatrix}
0&1&0&0&0&0&0&0&0\\
a q_1&0&0&1&0&0&0&0&0\\
0&0&0&0&1&0&0&0&0\\
0&c q_1&d q_1&0&0&0&0&0&0\\
0&e q_1&f q_1&0&0&0&1&0&0\\
0&g q_1&j q_1&0&0&0&0&1&0\\
u q_1^2+v q_1q_2&0&0&j q_1&f q_1&d q_1&0&0&0\\
w q_1^2+z q_1q_2&0&0&g q_1&e q_1&c q_1&0&0&1\\
0&w q_1^2+z q_1q_2&u q_1^2+v q_1q_2&0&0&0&0&a q_1&0
\end{pmatrix}.
\]
\[
M_{(X,H_2)}
=
\begin{pmatrix}
0&0&1&0&0&0&0&0&0\\
0&0&0&0&1&0&0&0&0\\
A q_2&0&0&0&0&1&0&0&0\\
0&J q_2&G q_2&0&0&0&1&0&0\\
0&F q_2&E q_2&0&0&0&0&1&0\\
0&D q_2&C q_2&0&0&0&0&0&0\\
W q_2^2+Z q_1q_2&0&0&C q_2&E q_2&G q_2&0&0&1\\
U q_2^2+V q_1q_2&0&0&D q_2&F q_2&J q_2&0&0&0\\
0&U q_2^2+V q_1q_2&W q_2^2+Z q_1q_2&0&0&0&A q_2&0&0
\end{pmatrix}.
\]

Quantum multiplication is commutative, so the commutator $[M_{H_1}, M_{H_2}]$ vanishes. This imposes further relations. For example, the (2,2)-entry of the commutator is $Jq_2 -eq_1$. As the $q_i$ are algebraically independent, we must have $J= e= 0$. We compute:
 
\[
e=g=j=0,
\qquad
E=G=J=0,
\]
\[
c=f=a,
\qquad
C=F=A,
\]
\[
u=w=0,
\qquad
U=W=0,
\]
and
\[
v=Ad,
\qquad
V=aD,
\qquad
z=Z=\frac{Dd}{2}.
\]

Thus, after imposing commutativity, the two matrices are controlled by only four
parameters:
\[
a,\ A,\ d,\ D.
\]

The next step is to extract information from the quantum period. We consider the matrix $M$ of the operator $(H_1+H_2) \star$ in the basis 
\[
\mathcal{B}' = \{ 1, H_1 +H_2, H_1^2 + H_2^2,  H_1 H_2,H_1^2H_2 + H_1H_2^2, H_1^2H_2^2, H_1 - H_2 ,H_1^2 - H_2^2, H_1^2H_2 - H_1H_2^2\}
\]
in terms of only $a,A,d,D$.
Let $M^+$ denote the $6\times 6$ submatrix, obtained by restricting to the first $6$ basis elements. We note that
\[
\scriptsize
M^{+}=
\begin{pmatrix}
0 & 1 & 0 & 0 & 0 & 0 \\[2pt]
a q_1 + A q_2 & 0 & 1 & 2 & 0 & 0 \\[2pt]
0 & \dfrac{(a+d)q_1+(A+D)q_2}{2} & 0 & 0 & 1 & 0 \\[8pt]
0 & \dfrac{a q_1+A q_2}{2} & 0 & 0 & 1 & 0 \\[8pt]
(Ad+aD+Dd)q_1q_2 & 0 & \dfrac{(a+d)q_1+(A+D)q_2}{2} & a q_1+A q_2 & 0 & 2 \\[8pt]
0 & \dfrac{(Ad+aD+Dd)q_1q_2}{2} & 0 & 0 & \dfrac{a q_1+A q_2}{2} & 0
\end{pmatrix}
\]

The matrix $M^+$ restricted to the diagonal Novikov slice \(q_1=q_2=q\) was computed in \cite{Fay26} using the quantum period. 
\[
M^+|_{q=q_1 =q_2} =
\begin{pmatrix}
0 & 1 & 0 & 0 & 0 & 0 \\
4q & 0 & 1 & 2 & 0 & 0 \\
0 & 6q & 0 & 0 & 1 & 0 \\
0 & 2q & 0 & 0 & 1 & 0 \\
32q^2 & 0 & 6q & 4q & 0 & 2 \\
0 & 16q^2 & 0 & 0 & 2q & 0
\end{pmatrix}.
\]
Comparing the matrices, 
we are left with the following set of equations.
\[
a+A=4,\qquad d+D=8,\qquad Ad+Da+Dd=32.
\]

The final step is to conclude that $a=A$ and $d=D$ from the geometry of the Verra fourfold.

\begin{lemma}
The Gromov--Witten invariants of $X$ are invariant under the natural symmetries $H_1 \leftrightarrow H_2$ and $L_1 \leftrightarrow L_2$ induced by swapping coordinates on $\P^2 \times \P^2$.
In particular,
\[
a= \langle H_1, H_1, pt\rangle_{L_1} = \langle H_2, H_2, pt\rangle _{L_2} = A,
\]
\[
d= \langle H_1, H_1^2, H_1^2H_2\rangle_{L_1} = \langle H_2, H_2^2 ,H_1H_2^2\rangle_{L_2} = D.
\]

\end{lemma}

\begin{proof}
Write a Verra fourfold as the double cover
\[
X_f \longrightarrow \P^2_x\times \mathbb \P^2_y
\]
branched over
\[
f_{2,2}(x,y)=0.
\]
The factor-swap
\[
(x,y)\longmapsto (y,x)
\]
sends \(X_f\) to the Verra fourfold \(X_{f^\sigma}\) branched over
\[
f^\sigma_{2,2}(x,y)=f_{2,2}(y,x).
\]
Under this operation,
\[
H_1\leftrightarrow H_2,
\qquad
q_1\leftrightarrow q_2.
\]

Since genus-zero Gromov--Witten invariants are deformation invariant in smooth families, and \(X_f\) and \(X_{f^\sigma}\) lie in the same smooth Verra family, these two invariants agree. 
\end{proof}
 
From the fact that
\[
-K_X = 2H_1 + 2H_2
\]
we have the following:
\begin{theorem}
\[
M(X)
=
\begin{pmatrix}
0&2&2&0&0&0&0&0&0\\
4q_1&0&0&2&2&0&0&0&0\\
4q_2&0&0&0&2&2&0&0&0\\
0&4q_1&8q_1&0&0&0&2&0&0\\
0&4q_2&4q_1&0&0&0&2&2&0\\
0&8q_2&4q_2&0&0&0&0&2&0\\
32q_1q_2&0&0&4q_2&4q_1&8q_1&0&0&2\\
32q_1q_2&0&0&8q_2&4q_2&4q_1&0&0&2\\
0&32q_1q_2&32q_1q_2&0&0&0&4q_2&4q_1&0
\end{pmatrix}.
\]
Its characteristic polynomial is
\[
\chi_X(\lambda)
=
\lambda^3
\left(
\lambda^6
-48(q_1+q_2)\lambda^4
+768(q_1^2-7q_1q_2+q_2^2)\lambda^2
-4096(q_1+q_2)^3
\right).
\]
In particular, X has naive atomic decomposition
\[
(3,1,1,1,1,1,1).
\]
\end{theorem}

\section{Complete intersection of two quadrics in $\P^6$}\label{sCIQuadrics}

Let $\sigma \colon Z \to \P^2$ be the blowup of $\P^2$ in eight general points $P_1, \dots, P_8$. Call the exceptional divisors $E_i$, $i=1, \dots , 8$. The linear system of quartics 
\[
|4H_{\P^2} - E_1 - \dots - E_7 - 2E_8|
\]
embeds $Z$ as a smooth surface of degree $5$ in $\P^4$, called a Castelnuovo surface \cite[Theorem 2.10 i)]{Ok83}. Let $g\colon Y\to \P^4$ be the blowup of $\P^4$ in $Z$ with exceptional divisor $E$. The linear system of cubics
\[
| 3H_{\P^4} - E | 
\]
defines a morphism $f \colon Y \to \P^6$ with image $X$ a complete intersection of two quadrics. Consider $E'$, the strict transform of the union of all three-secants to $Z$. The map $f$ contracts $E'$ to a line $\Lambda \subset X$. The class of $E'$ is $2H_{\P^4} -E$. Thus we have the diagram
\[
\xymatrix{
 & Y =\mathrm{Bl}_\Lambda (X) \simeq \mathrm{Bl}_Z (\P^4)\ar[ld]_f \ar[rd]^g & \\
 X & & \P^4 .
}
\]
Conversely, given a line $\Lambda$ in a complete intersection of two quadrics $X$, the projection from $\Lambda$ gives a birational morphism from the blowup $Y=\mathrm{Bl}_\Lambda (X)$ to $\P^4$ whose inverse has indeterminacy locus a Castelnuovo surface. 

The pencil of lines through $P_8$ in $\P^2$ gives $Z$ the structure of a conic bundle $\pi \colon Z \to \P^1$ with seven reducible fibres corresponding to the lines joining $P_8$ to one of the remaining seven blowup points. Therefore $E_1, \dots, E_7$ are one of the lines in one of these seven reducible fibres of the conic bundle, and $E_8$ is a section of the conic bundle intersecting the other lines in the seven reducible fibres. Let $C$ be the class of a general fibre of $\pi$. 

On $Y$ we then get the following algebraic classes: $H, P, L, pt$, the pullbacks to $Y$ of the class in $\P^4$ of a hyperplane, plane, line and point. Furthermore, surfaces $S_i$, the preimages on $Y$ of the $E_i$, and $U$ the preimage of $C$. There is also $Y$ itself and the class $E$ of the exceptional divisor of $g$. 

These classes generate the intersection ring of $Y$ as a $\Z$-module. For the relations among these defining the ring structure see \cite{BvBS26}. 

The cone of effective curves is generated by $L-3F$ ($Z$ has $3$-secants) and $F$. 

\begin{theorem}\label{tTwoQuadricsP6}
The matrix $M(Y)$ of quantum multiplication by the anticanonical $-K_Y$ with respect to the chosen basis of algebraic cycles is equal to
\[
\tiny{
\begin{array}{c|ccccc|ccccccccccc}
&pt&L&P&H& Y &-F&S_{1}&S_{2}&S_{3}&S_{4}&S_{5}&S_{6}&S_{7}&-U&-U-S_{8}&-E \\ 
\hline
      Y &0&5&0&0&0&-1&0&0&0&0&0&0&0&0&0&0\\
      H &0&0&5&0&0&0&1&1&1&1&1&1&1&-2&-4&0\\
      P&12qt^{2}&4qt^{3}&0&5&0&-2qt^{3}&0&0&0&0&0&0&0&0&0&-5\\
      L&20qt&42qt^{2}&4qt^{3}&0&5&-12qt^{2}&4qt^{3}&4qt^{3}&4qt^{3}&4qt^{3}&4qt^{3}&4qt^{3}&4qt^{3}&-8qt^{3}&-12qt^{3}&0\\
      pt & 5q&20qt&12qt^{2}&0&0&-4qt&6qt^{2}&6qt^{2}&6qt^{2}&6qt^{2}&6qt^{2}&6qt^{2}&6qt^{2}&-12qt^{2}&-21qt^{2}&0\\ 
      \hline
      E &0&0&5&0&0&-\frac{1}{t}&-1&-1&-1&-1&-1&-1&-1&2&3&0\\
      S_1&6qt^{2}&4qt^{3}&0&1&0&-2qt^{3}&-\frac{1}{t}&0&0&0&0&0&0&0&0&1\\
      S_2&6qt^{2}&4qt^{3}&0&1&0&-2qt^{3}&0&-\frac{1}{t}&0&0&0&0&0&0&0&1\\
      S_3&6qt^{2}&4qt^{3}&0&1&0&-2qt^{3}&0&0&-\frac{1}{t}&0&0&0&0&0&0&1\\
      S_4&6qt^{2}&4qt^{3}&0&1&0&-2qt^{3}&0&0&0&-\frac{1}{t}&0&0&0&0&0&1\\
      S_5&6qt^{2}&4qt^{3}&0&1&0&-2qt^{3}&0&0&0&0&-\frac{1}{t}&0&0&0&0&1\\
      S_6&6qt^{2}&4qt^{3}&0&1&0&-2qt^{3}&0&0&0&0&0&-\frac{1}{t}&0&0&0&1\\
      S_7 &6qt^{2}&4qt^{3}&0&1&0&-2qt^{3}&0&0&0&0&0&0&-\frac{1}{t}&0&0&1\\
      S_8 &9qt^{2}&4qt^{3}&0&2&0&-2qt^{3}&0&0&0&0&0&0&0&-\frac{1}{t}&0&1\\
      U & 12qt^{2}&8qt^{3}&0&2&0&-4qt^{3}&0&0&0&0&0&0&0&0&-\frac{1}{t}&2\\
      F &4qt&12qt^{2}&2qt^{3}&0&1&-3qt^{2}&2qt^{3}&2qt^{3}&2qt^{3}&2qt^{3}&2qt^{3}&2qt^{3}&2qt^{3}&-4qt^{3}&-6qt^{3}&-\frac{1}{t}
      \end{array}
}
\]
Its characteristic polynomial is
\begin{gather*}
\chi = t^{-11}
\left(\lambda t+1\right)^{9}
\left(6912q^{3}t^{10}-16q^{2}\lambda^{3}t^{8}+3600q^{2}\lambda^{2}t^{7}+3600q^{2}\lambda t^{6} 
-8q\lambda^{5}t^{5}
+492q\lambda^{4}t^{4}
\right.\\\left.
+4125q\lambda^{3}t^{3} 
+\left(-\lambda^{7}+9875q\lambda^{2}\right)t^{2}
+\left(-2\lambda^{6}+9375q\lambda\right)t
-\lambda^{5}+3125q
\right) 
\end{gather*}
and hence the naive atomic decomposition is
\[
    (9,1,1,1,1,1,1,1).
\]
The leading term with respect to $t$ is
\[
 \left(\lambda^{5}-3125q\right)
\]
recovering the $\P^4$ part.
\end{theorem}

\begin{proof}
The matrix $M(Y)$ and characteristic polynomial were computed in \cite{BvBS26} using the commutativity relations and the following Gromov--Witten invariants as input:
\begin{align*}
    \langle \mathrm{pt},\mathrm{pt} \rangle_{L}
        &= 1
        && \text{(one line through two points)}, \\
    \langle F,F \rangle_{L-2F}
        &= 1
        && \text{(one $2$-secant through two points on $Z$)}, \\
    \langle U,S_i \rangle_{F}
        &= 0, i \not=8
        && \text{(no intersections of different fibers in $Z\to \P^1$)}, \\
    \langle U,U \rangle_{F}
        &= 0
        && \text{(no intersections of different fibers in $Z\to \P^1$)}, \\
%    \langle U,S_8 \rangle_{F}
%        &= {\color{red} \pm1}
%        && \text{(one fibre over the intersection of a general conic and the section)}, \\
    \langle S_i,S_j \rangle_{F}
        &= 0, i\not=j
        && \text{(since $E_i.E_j=0$ for $i\not=j$)}, \\
     \langle U,F\rangle_{L-3F}
        &= 2
        && \\
\end{align*}

\noindent
Of these only the last one requires proof. Consider on $Z$ the pencil
\[
    |3H-E_1-\dots -E_8|.
\]
Its elements are plane cubics, that are embedded into $\P^4$ also as plane cubics. Indeed the intersection number with $|4H -E_1-\dots -E_7 - 3E_8|$ is $3$. Also, given one such cubic $G$ we find hyperplanes containing the image by considering $G+L_1$ and $G+L_2$
where $L_i$ are independent lines passing through $P_8$. Consider $\langle G \rangle$, the $\P^2$ spanned by 
the image of $G$. Every line in $\langle G \rangle$ is a $3$-secant to $Z$ and every $3$-secant arises in this way for some such $G$. To compute the Gromov--Witten invariant above notice that $Q := \sigma(F)$ is a point in $\P^2$. There is a unique cubic $G$ in the pencil considered above that contains $Q$. Recall also that $\sigma(U)$ is a line in $\P^2$ through $P_8$. This line meets $G$ in $3$ points one of which is $P_8$. Denote the other two intersection points by $Q_1$ and $Q_2$. Now look at the images of $Q,Q_1,Q_2$ in $\langle G \rangle$. The lines $\overline{Q_1Q}$ and $\overline{Q_2Q}$ spanned by their images are exactly the $3$-secants to $Z$ whose strict transforms intersect $F$ and $U$. 
\end{proof}

\begin{remark}
This shows that the naive variant of Iritani's theorem of Definition \ref{dIritaniMorePrecise} is false in general. Indeed, the naive atomic decomposition of $Z$ is $(1,1,1,1,1,1,1,1,1,1,1)$, since $Z$ is a del Pezzo surface of degree $1$, while the naive atomic decomposition of $\P^4$ is $(1,1,1,1,1)$. Nevertheless, the naive atomic decomposition of the blow-up of $\P^4$ along $Z$ does not consist of sixteen $1$’s. %Indeed, it is not difficult to check, using for example \cite{Beau97}, that the small variant of Iritani's theorem is also false in this case since all decompositions in the naive sense remain the same in the small sense {\color{red} This has to be made more precise}. 
\end{remark}

We now turn our attention to $f : Y \to X$.

\begin{proposition}
    Let $X$ be the complete intersection of two quadrics in $\mathbb{P}^6$ and $p$ the Novikov variable corresponding to a general line in $X$. Then the characteristic polynomial of the matrix representing quantum multiplication with $-K_X$ is
    $$
        \lambda^9(\lambda^3-432p)
    $$
    and hence the naive atomic decomposition of $X$ is
    $$
        (9,1,1,1).
    $$
\end{proposition}

\begin{proof}
    Setting $d_1=d_2=2$ in Beauville's Theorem \ref{tBeauville}, we obtain 
    $$
        2\sum_{i=1}^r (d_i - 1) - 1 = 2(1 + 1) - 1 = 3 \le 4 = \dim X
    $$
    so the theorem applies to our current situation. The Fano index of $X$ is $3$, and we obtain
    $$
        H^{\star 5} = 2^2 \cdot 2^2 H^{\star 2} 
    $$
    or equivalently
    $$
        (-K_X)^{\star 5} = 2^2 \cdot 2^2 \cdot 3^3 (-K_X)^{\star 2}.
    $$
    This gives a factor of
    $$
         \lambda^2(\lambda^3-432 p)
    $$
    for the characteristic polynomial of the matrix representing $-K_X \star$.
    
    The algebraic classes in degree $2$ are generated by $H^2$ and classes $Q_1,\dots,Q_7$ obtained geometrically as follows: In a generic pencil of quadrics in $\mathbb{P}^6$, there exist $7$ singular elements of rank $6$. Choose one ruling for each of these and a generic $\mathbb{P}^3$ contained in each such ruling. These $\mathbb{P}^3$'s intersect $X$ in quadric surfaces $Q_1,\dots,Q_7$. 
    The primitive classes $H^2-2Q_1, \dots, H^2-2Q_7$ contribute another factor of
    $$
        \lambda^7
    $$
    to the characteristic polynomial. 
\end{proof}

\begin{remark}
    This shows that the naive variant  of Iritanis theorem holds in this situation: Indeed the naive atomic decomposition for the blowup line is $(1,1)$. Since it has codimension 3 it counts twice, giving a combined atomic decomposition 
    \[
            (9,1,1,1)|(1,1)|(1,1) = (9,1,1,1,1,1,1,1).
    \]
%Again, the small decompositions are the same here. {\color{red} Make this more precise.}
\end{remark}

We now recover the characteristic polynomial of $M_{(X,-K_X)}$ from
that of $M_{(Y,-K_Y)}$ as a limit. 

\begin{lemma}
Let $N$ be a generic line in $X$ and $G$ a fiber of the fibration $E' \to \Lambda$. Set $p=q_M$ and $s = (q_G)^{-1}$.
Then $q = q_L = p^3s^2$ and $t = (q_F)^{-1} = (ps)^{-1}$.
\end{lemma}

\begin{proof}
Consider a generic line in $L \subset\P^4$. Its preimage under the projection from $\Lambda$ is a $\P^3$ that contains $\Lambda$. The $\P^3$ intersects $X$ in a complete intersection of two quadric surfaces. Generically such an intersection defines an elliptic curve of degree $4$ but here this intersection contains $\Lambda$. The residual curve is a rational normal curve $R$ of degree $3$. It intersects $\Lambda$ in $2$ points because the arithmetic genus of the union must be $1$. Hence
\[
    L = 3N-2G.
\]
Furthermore we see that $F$ is represented by a line on $X$ that intersects $\Lambda$. Indeed, it is such lines that are contracted by the projection from $\Lambda$. It follows that
\[
    F = N-G
\]
The claimed relations among the Novikov variables follow.     
\end{proof}

\begin{proposition}
    In the situation of Theorem \ref{tTwoQuadricsP6} consider the substitution
    \[
        q \mapsto p^3a^2, \quad \quad t \mapsto (ps)^{-1}.
    \]
    Then the leading term of $\chi$ with respect to $s$ is
    \[
        16\lambda^9(\lambda^3-432p)
    \]
    recovering the characteristic polynomial of $M_{(X,-K_X)}$.
\end{proposition}

\begin{proof}
    This is a straightforward computation, done for example in \cite{BvBS26}.
\end{proof}
\

\

%\begin{theorem}\label{tNaiveIritaniFails}
%In the above example, $X$ has naive atomic decomposition $(9,1,1,1)$ and $Y$ has naive atomic decomposition $(9,1,1,1,1,1,1,1)$. On the other hand, $Z$ has naive atomic decomposition consisting of a sequence of eleven $1$'s and $\P^4$ has of course $(1,1,1,1,1)$. Thus the naive variant of Iritani's theorem fails in this case. 
%\end{theorem}

\begin{proposition}\label{pSmallIritaniQuadrics}
The small atomic decomposition of $X$ is $(9,1,1,1)$ and hence agrees with the naive atomic decomposition. 
\end{proposition}

\begin{proof}
This can be computed using Theorem \ref{tBeauville}. 
\end{proof}

\begin{remark}\label{rSmall}
Since $Z$ and $\P^4$ have naive atomic decompositions consisting only of $1$'s, the same is true for the small atomic decompositions. Therefore, by Proposition \ref{pSmallIritaniQuadrics}, the small variant of Iritanis' theorem cannot hold for $f$ and $g$ simultaneously.
\end{remark}

\begin{remark}\label{rExplanationBigSmall}
In \cite[Example 2.13]{CKK25} the naive decomposition for $X$ is also mentioned, and it is pointed out that the correct atomic decomposition, looking at the eigenvalues of the big quantum multiplication by the Euler vector field, should consist of $12$ atoms with multiplicity one. This is confirmed in \cite{Hu21}. 

So for the big quantum cohomology ring Iritanis blowup formula holds since $\P^4$, $Y$ and $X$ all have only atoms with associated spaces of Hodge classes of dimension $1$. In the small quantum cohomology ring of $Y$ there are far fewer deformation parameters (namely, two) than in $\P^4$ and $Z$ together (i.e. 1 + 11). From this point of view, the small quantum cohomology ring of $Y$ is just a specialisation of the big quantum cohomology ring of $Y$ which is not generic enough. Geometrically the problem is that the del Pezzo surface $Z$ has a very interesting and complex cone of curves/Novikov ring of effective curve classes. However, most of these are just lines when viewed in the ambient $\P^4$. 
%Therefore, there are far fewer deformation parameters to deform the intersection ring with on $Y$ than there are on $Z$. In particular, deforming using only those curve classes, one does not recover the semisimple spectrum of $S$ on the blowup $Y$, but gets $(9,1,1,1)$ instead of the desired $11$. 
\end{remark}

\begin{remark}\label{rSoupedUpIritani}
In view of the preceding Remark \ref{rExplanationBigSmall}, it is perhaps reasonable to expect that the small variant of Iritani's theorem holds for all $(X, Z)$ if the inclusion $Z \hookrightarrow X$  induces an embedding $\mathrm{N}_1 (Z, \mathbb{Z}) \hookrightarrow \mathrm{N}_1 (X, \mathbb{Z})$. We know of no counterexample to this so far. If one replaces small by naive here, this becomes false, though, as the following example shows. We believe that this, however, is rather due to the fact that the right object to consider is the decomposition of $QH^{\bullet}_{\mathrm{alg}} (X)$ as an Artinian ring, and this is sometimes only imperfectly captured by the spectral decomposition of $M_{(X, -K_X)}$.  
\end{remark}

\begin{example}\label{eQuadric4}
 Let $Q$ be a four-dimensional smooth quadric. It is not hard to check that its small atomic decomposition is $(1,1,1,1,1,1)$ using Theorem \ref{tBeauville}, but its naive decomposition is $(2,1,1,1,1)$. Blowing up a plane in the quadric, the blowup has small and naive decompositions equal to $(1,1,1,1,1,1,1,1,1)$. The latter follows from \cite[Thm. 4.6]{HKLS25}. So the small variant of Iritani's theorem holds, but the naive one doesn't. 
\end{example}

%Unfortunately, the operator given by big quantum multiplication by the Euler vector field is in general rather hard to compute. The naive atomic decompositions we consider in this article provide partial information about its spectrum (the multiplicities of the eigenvalues of its spectrum are a refinement of the naive atomic decompositions we consider here). Still the naive decompositions are the expected true decompositions in many examples, and if the blowup centre $Z$ has a cone of curves generated by a single class, we surely expect the naive variant of Iritani's theorem to be true.

%\section{General observations}\label{sObservations}
\FloatBarrier

\end{document}